\documentclass[12pt,psamsfonts]{amsart}

\usepackage{amsmath,amssymb}
\usepackage{graphicx}
\usepackage{color}
\usepackage{psfrag} 
\usepackage{overpic} 

\newtheorem{theorem}{Theorem}
\newtheorem{proposition}[theorem]{Proposition}
\newtheorem{corollary}[theorem]{Corollary}

\newtheorem {remark}[theorem]{Remark}

\title[]{
Limit Cycles Bifurcating from a Periodic Annulus
in Discontinuous Planar Piecewise Linear Hamiltonian differential System with Three Zones}
\author[C. Pessoa and R. Ribeiro]{}


\begin{document}
 \maketitle

\centerline{\scshape  Claudio Pessoa and Ronisio Ribeiro}
\medskip

{\footnotesize \centerline{Universidade Estadual Paulista (UNESP),} \centerline{Instituto de Bioci\^encias Letras e Ci\^encias Exatas,} \centerline{R. Cristov\~ao Colombo, 2265, 15.054-000, S. J. Rio Preto, SP, Brazil }
\centerline{\email{c.pessoa@unesp.br} and \email{ronisio.ribeiro@unesp.br}}}

\medskip

\bigskip

\begin{quote}{\normalfont\fontsize{8}{10}\selectfont
{\bfseries Abstract.}
In this paper, we study the number of limit cycles that can bifurcating from a periodic annulus in discontinuous planar piecewise linear Hamiltonian differential system with three zones separated by two parallel straight lines. We prove that if the central subsystem, i.e. the system defined between the two parallel lines, has a real center and the others subsystems have centers or saddles, then we have at least three limit cycles that appear after perturbations of periodic annulus. For this, we study the number of zeros of a Melnikov function for piecewise Hamiltonian system and present a normal form for this system in order to simplify the computations.
\par}
\end{quote}

\section{Introduction and Main Results}
The first works on piecewise differential systems appeared in the 1930s, see  \cite{And66}. This class of systems have great applicability, mainly in mechanics, electrical circuits, control theory, etc (see for instance the book \cite{diB08} and the papers \cite{Chu90, Fit61, McK70, Nag62}). This subject has piqued the attention of researchers in qualitative theory of differential equations and numerous studies about this topic have arisen in the literature recently.

Piecewise differential systems with two zones are the most studied, either for their applications in modeling phenomena in general or for their apparent simplicity (see \cite{Fre12, LiS19b}). As in the smooth case, the researches are mainly concentrated to the determination of the number and position of the limit cycles of these systems. In 1998, Freire, Ponce, Rodrigo and Torres in \cite{Fre98} proved that a continuous piecewise linear differential systems in the plane with two zones has at most one limit cycle. In the discontinuous case, the maximum number of limit cycles is not known, but important partial results about this problem have been obtained, see for example \cite{Buz13, Bra13, Fre14b, Lli18b}. 

The problem becomes more complicated when we have more than two zones and there are a few works that deal with the discontinuous case (see \cite{Don17, Hu13, Lli15b, Wan16}). However, when restrictive hypotheses such as symmetry and linearity are imposed, the issue of limit cycles is well explored. More precisely, for symmetric continuous piecewise linear differential systems with three zones, conditions for nonexistence and existence of one, two or three limit cycles have been obtained (see for instance the book \cite{Lli14}). For the nonsymmetric case, examples with two limit cycles surrounding the only singular point at the origin was found in \cite{Lim17, Lli15}.

Recently, some researchers have been trying to estimate the number of limit cycles in discontinuous piecewise Hamiltonian differential systems with three zones. In this direction, we have papers with one limit cycle, see \cite{Fon20, Lli18a} and more than two limit cycles, see \cite{Xio20, Xio21, Yan20}. In this work, we contribute along these lines. Our goal is to study the number of limit cycles that can bifurcated from periodic annulus of families of discontinuous planar piecewise linear Hamiltonian differential system with three zones separated by two parallel straight lines. We prove that if the central subsystem, i.e. the system between the two parallel lines, has a real center and the others subsystems have  centers or saddles, then we have at least three limit cycles,  visiting the three zones, that bifurcate from an periodic annulus. Our results are obtained by studying the number of zeros of the Melnikov function for piecewise Hamiltonian system, see the papers \cite{Xio20, Xio21} for more details about the Melnikov function.

In order to set the problem, let $h_i:\mathbb{R}^2\rightarrow\mathbb{R}$,  $i=L,R$, be the functions  $h_{\scriptscriptstyle L}(x,y)=x+1$ and $h_{\scriptscriptstyle R}(x,y)=x-1$. Denote by $\Sigma_{\scriptscriptstyle L}=h_{\scriptscriptstyle L}^{-1}(0)$ and $\Sigma_{\scriptscriptstyle R}=h_{\scriptscriptstyle R}^{-1}(0)$ the {\it switching curves}. This straight lines decomposes the plane in three regions 
$$R_{\scriptscriptstyle L}=\{(x,y)\in\mathbb{R}^2:x<-1\},\quad R_{{\scriptscriptstyle C}}=\{(x,y)\in\mathbb{R}^2:-1<x<1\},$$ 
and 
$$R_{{\scriptscriptstyle R}}=\{(x,y)\in\mathbb{R}^2:x>1\}.$$

Consider the discontinuous planar piecewise linear near--Hamiltonian system with three zones, given by
\begin{equation}\label{eq:01}
	\left\{\begin{array}{ll}
		\dot{x}= H_y(x,y)+\epsilon f(x,y), \\
		\dot{y}= -H_x(x,y)+\epsilon g(x,y),
	\end{array}
	\right.
\end{equation}
with
\begin{equation*}
	H(x,y)=\left\{\begin{array}{ll}\vspace{0.2cm}
		\hspace{-0.3cm}H^{\scriptscriptstyle L}(x,y)=\dfrac{b_{\scriptscriptstyle L}}{2}y^2-\dfrac{c_{\scriptscriptstyle L}}{2}x^2+a_{\scriptscriptstyle L}xy+\alpha_{\scriptscriptstyle L}y-\beta_{\scriptscriptstyle L}x, \quad x\leq -1, \\ \vspace{0.2cm}
		\hspace{-0.3cm}H^{\scriptscriptstyle C}(x,y)=\dfrac{b_{\scriptscriptstyle C}}{2}y^2-\dfrac{c_{\scriptscriptstyle C}}{2}x^2+a_{\scriptscriptstyle C}xy+\alpha_{\scriptscriptstyle C}y-\beta_{\scriptscriptstyle C}x, -1\leq x\leq 1, \\
		\hspace{-0.3cm}H^{\scriptscriptstyle R}(x,y)=\dfrac{b_{\scriptscriptstyle R}}{2}y^2-\dfrac{c_{\scriptscriptstyle R}}{2}x^2+a_{\scriptscriptstyle R}xy+\alpha_{\scriptscriptstyle R}y-\beta_{\scriptscriptstyle R}x, \quad x \geq 1, \\
	\end{array}
	\right.
\end{equation*}
\begin{equation}\label{eq:02}
	f(x,y)=\left\{\begin{array}{ll}
		f_{\scriptscriptstyle L}(x,y)=p_{\scriptscriptstyle L}x+q_{\scriptscriptstyle L}y+r_{\scriptscriptstyle L}, \quad x\leq -1, \\
		f_{\scriptscriptstyle C}(x,y)=p_{\scriptscriptstyle C}x+q_{\scriptscriptstyle C}y+r_{\scriptscriptstyle C}, \quad -1\leq x\leq 1, \\
		f_{\scriptscriptstyle R}(x,y)=p_{\scriptscriptstyle R}x+q_{\scriptscriptstyle R}y+r_{\scriptscriptstyle R}, \quad x \geq 1, \\
	\end{array}
	\right.
\end{equation}
\begin{equation}\label{eq:03}
	g(x,y)=\left\{\begin{array}{ll}
		g_{\scriptscriptstyle L}(x,y)=s_{\scriptscriptstyle L}x+u_{\scriptscriptstyle L}y+v_{\scriptscriptstyle L}, \quad x\leq -1, \\
		g_{\scriptscriptstyle C}(x,y)=s_{\scriptscriptstyle C}x+u_{\scriptscriptstyle C}y+v_{\scriptscriptstyle C}, \quad -1\leq x\leq 1, \\
		g_{\scriptscriptstyle R}(x,y)=s_{\scriptscriptstyle R}x+u_{\scriptscriptstyle R}y+v_{\scriptscriptstyle R}, \quad x \geq 1, \\
	\end{array}
	\right.
\end{equation}
where the dot denotes the derivative with respect to the independent variable $t$, here called the time, and $0\leq\epsilon<<1$. We call system \eqref{eq:01} of {\it left subsystem} when $x\leq -1$, {\it right subsystem} when $x\geq 1$ and {\it central subsystem} when $-1\leq x\leq 1$. Denote by $X_{\scriptscriptstyle L}(x,y)$, $X_{\scriptscriptstyle C}(x,y)$ and $X_{\scriptscriptstyle R}(x,y)$ the planar piecewise linear vector fields associated with the left, central and right subsystem from  $\eqref{eq:01}|_{\epsilon=0}$, respectively. 

We will use the vector field $X_{\scriptscriptstyle L}$ and the switching curve $\Sigma_{\scriptscriptstyle L}$ in the next definitions. However, we can easily adapt the definitions to the vector fields $X_{\scriptscriptstyle C}$ and $X_{\scriptscriptstyle R}$ and the switching curve $\Sigma_{\scriptscriptstyle R}$.

We say that the vector field $X_{\scriptscriptstyle L}$ has a real equilibrium $p$ if $p$ is an equilibrium of $X_{\scriptscriptstyle L}$ and $p\in R_{\scriptscriptstyle L}$. Otherwise, we will say that $X_{\scriptscriptstyle L}$ has a virtual equilibrium $p$ if $p\in (R_{\scriptscriptstyle L})^c$, where $(R_{\scriptscriptstyle L})^c$ denotes the complementary of $R_{\scriptscriptstyle L}$ in $\mathbb{R}^2$.

The derivative of function $h_{\scriptscriptstyle L}$ in the direction of the vector field $X_{\scriptscriptstyle L}$, i.e., the expression
$X_{\scriptscriptstyle L} h_{\scriptscriptstyle L}(p)=\langle  X_{\scriptscriptstyle L}(p),\nabla h_{\scriptscriptstyle L}(p)\rangle,$
where $\left\langle \cdot,\cdot\right\rangle$ is the usual inner product in $\mathbb{R}^2$,  characterize the contact between the vector field $X_{\scriptscriptstyle L}$ and the switching curve $\Sigma_{\scriptscriptstyle L}$. 
When $p\in\Sigma_{\scriptscriptstyle L}$ and $X_{\scriptscriptstyle L} h_{\scriptscriptstyle L}(p)=0$ we say that $p$ is a {\it tangent point} of $X_{\scriptscriptstyle L}$.
We distinguish the followings subsets of $\Sigma_{\scriptscriptstyle L}$ (the same for $\Sigma_{\scriptscriptstyle R}$).

\medskip

\noindent Crossing set:
$$\Sigma_{\scriptscriptstyle L}^{c}=\{p\in \Sigma_{\scriptscriptstyle L}:X_{\scriptscriptstyle L} h_{\scriptscriptstyle L}(p) \cdot X_{\scriptscriptstyle C} h_{\scriptscriptstyle L}(p)>0\};$$

\noindent Sliding set:
$$\Sigma_{\scriptscriptstyle L}^{s}=\{p\in \Sigma_{\scriptscriptstyle L}:X_{\scriptscriptstyle L} h_{\scriptscriptstyle L}(p)>0, X_{\scriptscriptstyle C} h_{\scriptscriptstyle L}(p)<0\};$$

\noindent Escaping set:
$$\Sigma_{\scriptscriptstyle L}^{e}=\{p\in \Sigma_{\scriptscriptstyle L}:X_{\scriptscriptstyle L} h_{\scriptscriptstyle L}(p)<0, X_{\scriptscriptstyle C} h_{\scriptscriptstyle L}(p)>0\}.$$


Suppose that system $\eqref{eq:01}|_{\epsilon=0}$ satisfies the following hypotheses:
\begin{itemize}
	\item[(H1)] The unperturbed central subsystem from $\eqref{eq:01}|_{\epsilon=0}$ has a real center and the others unperturbed subsystems from $\eqref{eq:01}|_{\epsilon=0}$ have centers or saddles.
	\item[(H2)] The unperturbed system from $\eqref{eq:01}|_{\epsilon=0}$ has only crossing points on the straights lines $x=\pm 1$, except by some tangent points.
	\item[(H3)] The unperturbed system from $\eqref{eq:01}|_{\epsilon=0}$ has a periodic annulus consisting of a family of crossing periodic orbits around the origin such that each orbit of this family passes thought the three zones with clockwise orientation. 	
\end{itemize}

The main result in this paper is the follow.

\begin{theorem}\label{the:01}
	The number of limit cycles of system \eqref{eq:01}, satisfying hypothese {\rm (Hi)} for $i=1,2,3$, which can bifurcate from the periodic annulus of the unperturbed system $\eqref{eq:01}|_{\epsilon=0}$ is at least three.
\end{theorem}

The paper is organized as follows. In Section \ref{sec:Mel} we introduce the first order Melnikov function associated to system $\eqref{eq:01}$. In Section \ref{sec:NF} we obtain a normal form to system $\eqref{eq:01}|_{\epsilon=0}$ that simplifies the computations and in Section \ref{sec:Teo} we will prove Theorem \ref{the:01}.

\section{Melnikov Function}\label{sec:Mel}

In this section, we will present the first order Melnikov function associated to system $\eqref{eq:01}$ that we will use to prove the main result of this paper.

Suppose that $\eqref{eq:01}|_{\epsilon=0}$ satisfies the hypothesis (H3), i.e. there exists an open interval $J=(\alpha,\beta)$ such that for each $h\in J$ we have four points, $A(h)=(1,h)$, $A_1(h)=(1,a_1(h))\in \Sigma_{\scriptscriptstyle R}$, with $a_1(h)<h$, and $A_2(h)=(-1,a_2(h))$, $A_3(h)=(-1,a_3(h))\in \Sigma_{\scriptscriptstyle L}$, with $a_2(h)<a_3(h)$, whose are determined by the following equations
\begin{equation}\label{eq:05}
	\begin{aligned}
		& H^{\scriptscriptstyle R}(A(h))=H^{\scriptscriptstyle R}(A_1(h)), \\
		& H^{\scriptscriptstyle C}(A_1(h))=H^{\scriptscriptstyle C}(A_2(h)), \\
		& H^{\scriptscriptstyle L}(A_2(h))=H^{\scriptscriptstyle L}(A_3(h)), \\
		& H^{\scriptscriptstyle C}(A_3(h))=H^{\scriptscriptstyle C}(A(h)), 
	\end{aligned}
\end{equation}
satisfying, for $h\in J$,
$$H^{\scriptscriptstyle R}_y(A(h))\,H^{\scriptscriptstyle R}_y(A_1(h))\,H^{\scriptscriptstyle L}_y(A_2(h))\,H^{\scriptscriptstyle L}_y(A_3(h))\ne 0,$$ 
$$H^{\scriptscriptstyle C}_y(A(h))\,H^{\scriptscriptstyle C}_y(A_1(h))\,H^{\scriptscriptstyle C}_y(A_2(h))\,H^{\scriptscriptstyle C}_y(A_3(h))\ne 0.$$
Moreover, system $\eqref{eq:01}|_{\epsilon=0}$ has a crossing periodic orbit $L_h=L_h^{\scriptscriptstyle R}\cup\bar{L}_h^{\scriptscriptstyle C}\cup L_h^{\scriptscriptstyle L}\cup L_h^{\scriptscriptstyle C}$ passing through these points (see Fig. \ref{fig:01}), where
\begin{equation*}
	\begin{aligned}
		 L_h^{\scriptscriptstyle R}=\,&\Big\{(x,y)\in\mathbb{R}^2:H^{\scriptscriptstyle R}(x,y)=H^{\scriptscriptstyle R}(A(h))=\dfrac{b_{\scriptscriptstyle R}}{2}h^2+(a_{\scriptscriptstyle R}+\alpha_{\scriptscriptstyle R})h\\
		&\,\,\,\,-\Big(\dfrac{c}{2}+\beta_{\scriptscriptstyle R}\Big),x>1\Big\}, \\
		 \bar{L}_h^{\scriptscriptstyle C}=\,&\{(x,y)\in\mathbb{R}^2:H^{\scriptscriptstyle C}(x,y)=H^{\scriptscriptstyle C}(A_1(h)),-1\leq x\leq1 \quad\text{and}\quad y<0\}, \\
		 L_h^{\scriptscriptstyle L}=\,&\{(x,y)\in\mathbb{R}^2:H^{\scriptscriptstyle L}(x,y)=H^{\scriptscriptstyle L}(A_2(h)),x<-1\}, \\
		 L_h^{\scriptscriptstyle C}=\,&\{(x,y)\in\mathbb{R}^2:H^{\scriptscriptstyle C}(x,y)=H^{\scriptscriptstyle C}(A_3(h)),-1\leq x\leq1 \quad\text{and}\quad y>0\}.
	\end{aligned}
\end{equation*}

\begin{figure}[h]
	\begin{center}		
		\begin{overpic}[width=3in]{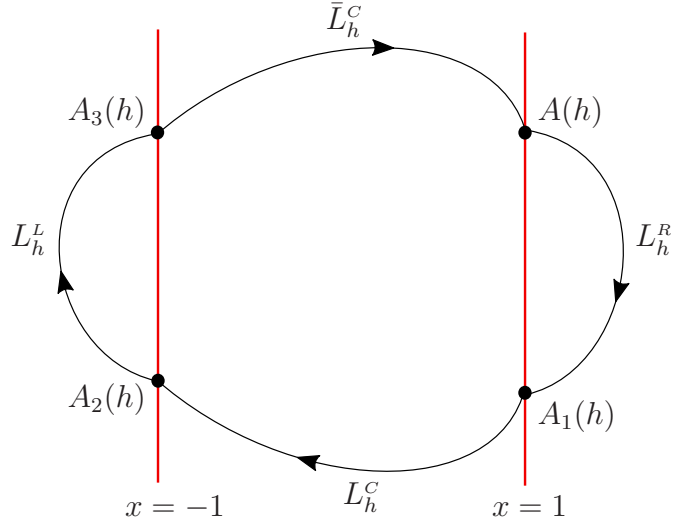}
			\put(76,-5) {$x=1$}
			\put(12,-5) {$x=-1$}
			\put(84,64) {$A(h)$}
			\put(84,11) {$A_1(h)$}
			\put(2,14) {$A_2(h)$}
			\put(2,64) {$A_3(h)$}
			\put(101,42) {$L_h^{\scriptscriptstyle R}$}
			\put(-8,42) {$L_h^{\scriptscriptstyle L}$}
			\put(50,-3) {$L_h^{\scriptscriptstyle C}$}
			\put(47,80) {$\bar{L}_h^{\scriptscriptstyle C}$}
		\end{overpic}
	\end{center}
	\vspace{0.7cm}
	\caption{The crossing periodic orbit of system  $\eqref{eq:01}|_{\epsilon=0}$.}\label{fig:01}
\end{figure} 

Consider the solution of right subsystem from \eqref{eq:01} starting from point $A(h)$. Let $A_{\epsilon}(h)=(1,a_{\epsilon}(h))$ be the first intersection point of this orbit with straight line $x=1$. Denote by $B_{\epsilon}(h)=(-1,b_{\epsilon}(h))$ the first intersection point of the orbit from central subsystem from \eqref{eq:01} starting at $A_{\epsilon}(h)$ with straight line $x=-1$, $C_{\epsilon}(h)=(-1,c_{\epsilon}(h))$ the first intersection point of the orbit from left subsystem from \eqref{eq:01} starting at $B_{\epsilon}(h)$ with straight line $x=-1$ and $D_{\epsilon}(h)=(1,d_{\epsilon}(h))$ the first intersection point of the orbit from central subsystem from \eqref{eq:01} starting at $C_{\epsilon}(h)$ with straight line $x=1$ (see Fig. \ref{fig:02}).

\begin{figure}[h]
	\begin{center}		
		\begin{overpic}[width=3in]{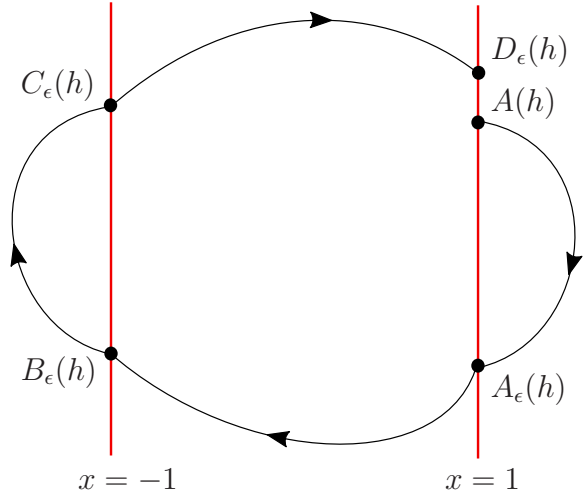}
			\put(76,-5) {$x=1$}
			\put(12,-5) {$x=-1$}
			\put(84,61) {$A(h)$}
			\put(84,11) {$A_\epsilon(h)$}
			\put(2,14)  {$B_\epsilon(h)$}
			\put(2,64)  {$C_\epsilon(h)$}
			\put(84,70)  {$D_\epsilon(h)$}
		\end{overpic}
	\end{center}
	\vspace{0.7cm}
	\caption{Poincaré map of system  \eqref{eq:01}.}\label{fig:02}
\end{figure} 

We define the Poincaré map of piecewise system \eqref{eq:01} as follows, 
$$H^{\scriptscriptstyle R}(D_{\epsilon}(h))-H^{\scriptscriptstyle R}(A(h))=\epsilon M(h)+\mathcal{O}(\epsilon^2),$$
where $M(h)$ is called the {\it first order Melnikov function} associated to piecewise system \eqref{eq:01}. Then, using the same idea of Theorem 1.1 in \cite{Liu10}, it is easy to obtain the following theorem.

\begin{theorem}
	Consider  system \eqref{eq:01} with $0\leq \epsilon <<1$ and suppose that the unperturbed system $\eqref{eq:01}|_{\epsilon=0}$ has a family of crossing periodic orbits around the origin. Then the first order Melnikov function can be expressed as
	\begin{equation}\label{eq:mel}
		\begin{aligned}
			M(h) & = \frac{H_y^{\scriptscriptstyle R}(A)}{H_y^{\scriptscriptstyle C}(A)} I_{\scriptscriptstyle C} + \frac{H_y^{\scriptscriptstyle R}(A)H_y^{\scriptscriptstyle C}(A_3)}{H_y^{\scriptscriptstyle C}(A)H_y^{\scriptscriptstyle L}(A_3)} I_{\scriptscriptstyle L}\\
			& \quad + \frac{H_y^{\scriptscriptstyle R}(A)H_y^{\scriptscriptstyle C}(A_3)H_y^{\scriptscriptstyle L}(A_2)}{H_y^{\scriptscriptstyle C}(A)H_y^{\scriptscriptstyle L}(A_3)H_y^{\scriptscriptstyle C}(A_2)} \bar{I}_{\scriptscriptstyle C} \\
			&\quad  +  \frac{H_y^{\scriptscriptstyle R}(A)H_y^{\scriptscriptstyle C}(A_3)H_y^{\scriptscriptstyle L}(A_2)H_y^{\scriptscriptstyle C}(A_1)}{H_y^{\scriptscriptstyle C}(A)H_y^{\scriptscriptstyle L}(A_3)H_y^{\scriptscriptstyle C}(A_2)H_y^{\scriptscriptstyle R}(A_1)} I_{\scriptscriptstyle R},
		\end{aligned}
	\end{equation}  
	where
	$$I_{\scriptscriptstyle C}=\int_{\widehat{A_3A}}g_{\scriptscriptstyle C}dx-f_{\scriptscriptstyle C}dy, \,\, I_{\scriptscriptstyle L}=\int_{\widehat{A_2A_3}}g_{\scriptscriptstyle L}dx-f_{\scriptscriptstyle L}dy,\,\, \bar{I}_{\scriptscriptstyle C}=\int_{\widehat{A_1A_2}}g_{\scriptscriptstyle C}dx-f_{\scriptscriptstyle C}dy$$
	and
	$$\quad I_{\scriptscriptstyle R}=\int_{\widehat{AA_1}}g_{\scriptscriptstyle R}dx-f_{\scriptscriptstyle R}dy.$$
	Furthermore, if $M(h)$ has a simple zero at $h^{*}$, then for $0< \epsilon <<1$, the system \eqref{eq:01} has a unique limit cycle near $L_{h^{*}}$. 
\end{theorem}	

\section{Normal Form}\label{sec:NF}

In order to simplify the computations to prove Theorem \ref{the:01} is convenient to do a continuous linear change of variables which transform system $\eqref{eq:01}|_{\epsilon=0}$ into a simple form. This change of variables is a homeomorphism with keeps invariant the straight lines $x=\pm 1$. Furthermore, this homeomorphism will be a topological equivalence between the systems. More precisely, we have the follow result.

\begin{proposition}\label{fn:01}
	The discontinuous piecewise linear differential systems $\eqref{eq:01}|_{\epsilon=0}$ satisfying assumption (Hi), $i=1,2,3$, after a change of variables can be written as
	\begin{equation}\label{eq:04}
		\left\{\begin{array}{ll}
			\dot{x}= H_y(x,y), \\
			\dot{y}= -H_x(x,y),
		\end{array}
		\right.
	\end{equation}
	where
	\begin{equation}\label{eq:fnh}
		H(x,y)=\left\{\begin{array}{ll}\vspace{0.2cm}
			\begin{aligned}
			H^{\scriptscriptstyle L}(x,y)=&\,\dfrac{b_{\scriptscriptstyle L}}{2}y^2-\dfrac{c_{\scriptscriptstyle L}}{2}x^2+a_{\scriptscriptstyle L}xy\\
			&+a_{\scriptscriptstyle L}y-\beta_{\scriptscriptstyle L}x,\quad\quad\quad\quad\quad x\leq -1, \\ \vspace{0.2cm}
			H^{\scriptscriptstyle C}(x,y)=&\,\dfrac{1}{2}x^2+\dfrac{1}{2}y^2, \quad\quad\quad\quad -1\leq x\leq 1, \\
			H^{\scriptscriptstyle R}(x,y)=&\,\dfrac{b_{\scriptscriptstyle R}}{2}y^2-\dfrac{c_{\scriptscriptstyle R}}{2}x^2+a_{\scriptscriptstyle R}xy\\
			&-a_{\scriptscriptstyle R}y-\beta_{\scriptscriptstyle R}x,\,\, \quad\quad\quad\quad\quad x \geq 1. \\
			\end{aligned}
		\end{array}
		\right.
	\end{equation}
\end{proposition}
\begin{proof}
	Through a translation, we can assume that the singularity of the central subsystem from $\eqref{eq:01}|_{\epsilon=0}$ is the origin, i.e. $\alpha_{\scriptscriptstyle C}=\beta_{\scriptscriptstyle C}=0$.
	By the hypotheses (H1) and (H3), we have that the central subsystem from $\eqref{eq:01}|_{\epsilon=0}$ satisfy $a_{\scriptscriptstyle C}^2+b_{\scriptscriptstyle C}c_{\scriptscriptstyle C}<0$ and $b_{\scriptscriptstyle C}>0$. Note that $b_i\ne 0$, for $i=L,R$. In fact, if the singular points of the subsystems from $\eqref{eq:01}|_{\epsilon=0}$ are centers, this is true due to the clockwise orientation of the orbits. Now, if the singular points are saddles and $b_i=0$ we have a separatrices parallel to switching straight lines $x=\pm 1$. System $\eqref{eq:01}|_{\epsilon=0}$ has four tangent points given by $P_1=(1,-a_{\scriptscriptstyle C}/b_{\scriptscriptstyle C})$, $P_2=(1,-(a_{\scriptscriptstyle R}+\alpha_{\scriptscriptstyle R})/b_{\scriptscriptstyle R})$, $P_3=(-1,a_{\scriptscriptstyle C}/b_{\scriptscriptstyle C})$ and $P_4=(-1,(a_{\scriptscriptstyle L}-\alpha_{\scriptscriptstyle L})/b_{\scriptscriptstyle L})$. By hypothesis (H2), we have that the system $\eqref{eq:01}|_{\epsilon=0}$ have only crossing points on the straight lines $x=\pm 1$, except in the tangent points. Hence, for all $y\in\mathbb{R}\setminus\{\pm a_{\scriptscriptstyle C}/b_{\scriptscriptstyle C},-(a_{\scriptscriptstyle R}+\alpha_{\scriptscriptstyle R})/b_{\scriptscriptstyle R}),(a_{\scriptscriptstyle L}-\alpha_{\scriptscriptstyle L})/b_{\scriptscriptstyle L})\}$, we must have
	$$\left\langle  X_{\scriptscriptstyle L}(-1,y),(1,0)\right\rangle \left\langle  X_{\scriptscriptstyle C}(-1,y),(1,0)\right\rangle>0$$
	and
	$$\left\langle  X_{\scriptscriptstyle R}(1,y),(1,0)\right\rangle \left\langle  X_{\scriptscriptstyle C}(1,y),(1,0)\right\rangle>0.$$ 
	But this implies that $b_{\scriptscriptstyle L}b_{\scriptscriptstyle C}>0$, $b_{\scriptscriptstyle R}b_{\scriptscriptstyle C}>0$, $P_1=P_2$ and $P_3=P_4$. Therefore, as $b_{\scriptscriptstyle C}>0$, we have that
	\begin{equation}\label{ch:01}
		\alpha_{\scriptscriptstyle L}=\frac{a_{\scriptscriptstyle L}b_{\scriptscriptstyle C}-a_{\scriptscriptstyle C}b_{\scriptscriptstyle L}}{b_{\scriptscriptstyle C}},\quad b_{\scriptscriptstyle L}>0,\quad \alpha_{\scriptscriptstyle R}=\frac{-a_{\scriptscriptstyle R}b_{\scriptscriptstyle C}+a_{\scriptscriptstyle C}b_{\scriptscriptstyle R}}{b_{\scriptscriptstyle C}}\quad\text{and}\quad b_{\scriptscriptstyle R}>0.
	\end{equation}
	Assuming the conditions \eqref{ch:01}, consider the change of variables  
	\begin{displaymath}
		\left(\begin{array}{c}
			x\\
			y
		\end{array}\right)=\left(\begin{array}{cc}
			1 & 0\\
			-\dfrac{a_{\scriptscriptstyle C}}{b_{\scriptscriptstyle C}} & \dfrac{\sqrt{-a^2_{\scriptscriptstyle C}-b_{\scriptscriptstyle C}c_{\scriptscriptstyle C}}}{b_{\scriptscriptstyle C}}
		\end{array}\right)\left(\begin{array}{c}
			u\\
			v
		\end{array}\right)
	\end{displaymath}
	and rescaling the time by $\tilde{t}=\sqrt{-a^2_{\scriptscriptstyle C}-b_{\scriptscriptstyle C}c_{\scriptscriptstyle C}}\,t$.
	Applying the change of variables and rescaling the time above and rewriting the parameters conveniently, system $\eqref{eq:01}|_{\epsilon=0}$ becomes system \eqref{eq:04}.	
\end{proof}

\medskip

In what follows, we will consider the discontinuous planar piecewise linear near--Hamiltonian system \eqref{eq:01} 
with $f(x,y)$, $g(x,y)$ and $H(x,y)$ given by \eqref{eq:02}, \eqref{eq:03} and \eqref{eq:fnh}, respectively.

\section{Proof of Theorem \ref{the:01}}\label{sec:Teo}

The proof of Theorem \ref{the:01} is a straightforward consequence of Corollarys \ref{scs-a}, \ref{ccs-c}-\ref{ccc-c}. 

We can classify the systems that satisfy the hypothesis (H1) according to the configuration of their singular points. Thus, denoting the centers by the capital letter C and by S the saddles, in the case of three zones, we have the following three class of piecewise linear Hamiltonian systems: SCS, CCS and CCC. This is, CCC indicates that the singular points of
the linear systems that define the piecewise differential system are centers and so on.

In order to computate the zeros of the first order Melnikov function, it is necessary to find the open interval $J$, where it is define. For this, consider the follow proposition.

\begin{proposition}
	Consider the system \eqref{eq:01} with the hypotheses (Hi), $i=1,2,3$. 
	\begin{itemize}
		\item[(a)] If the system $\eqref{eq:01}|_{\epsilon=0}$ is of type SCS or CCS, then  $J=(0,\tau)$, where $\tau=(a_{\scriptscriptstyle R}^2-b_{\scriptscriptstyle R}\beta_{\scriptscriptstyle R}-\omega_{\scriptscriptstyle RS}^2)/b_{\scriptscriptstyle R}\omega_{\scriptscriptstyle RS}$ with $\omega_{\scriptscriptstyle RS}=\sqrt{a^2_{\scriptscriptstyle R} + b_{\scriptscriptstyle R}c_{\scriptscriptstyle R}}$, and the periodic annulus are equivalents to one of the figures of Fig. \ref{fig:03}.  
		\item[(b)] If the system $\eqref{eq:01}|_{\epsilon=0}$ is of type CCC, then $J=(0,\infty)$, and the periodic annulus are equivalents to one of the figures of Fig. \ref{fig:04}.
	\end{itemize}
\end{proposition}
\begin{proof}
	Suppose that the system $\eqref{eq:01}|_{\epsilon=0}$ is of type SCS or CCS. Note  that if the saddles are virtual or if they are under the straight lines $x=\pm 1$, then we have not periodic orbits passing through the three zones.
	Denote by $W^u_{\scriptscriptstyle R}$ and $W^s_{\scriptscriptstyle R}$ (resp. $W^u_{\scriptscriptstyle L}$ and $W^s_{\scriptscriptstyle L}$) the unstable and stable separatrices of the saddles of the right (resp. left) subsystems from $\eqref{eq:01}|_{\epsilon=0}$, respectively. Denote by $P_{\scriptscriptstyle L}^{i}=W^i_{\scriptscriptstyle L}\cap \Sigma_{\scriptscriptstyle L}$ and $P_{\scriptscriptstyle R}^{i}=W^i_{\scriptscriptstyle R}\cap \Sigma_{\scriptscriptstyle R}$, for $i=u,s$. After some computate, is possible to show that
	$$P_{\scriptscriptstyle L}^{u}=\Bigg(-1,\frac{a_{\scriptscriptstyle L}^2+b_{\scriptscriptstyle L} \beta_{\scriptscriptstyle L}-\omega_{\scriptscriptstyle LS}^2}{b_{\scriptscriptstyle L}\omega_{\scriptscriptstyle LS}}\Bigg),\quad P_{\scriptscriptstyle L}^{s}=\Bigg(-1,-\frac{a_{\scriptscriptstyle L}^2+b_{\scriptscriptstyle L} \beta_{\scriptscriptstyle L}-\omega_{\scriptscriptstyle LS}^2}{b_{\scriptscriptstyle L}\omega_{\scriptscriptstyle LS}}\Bigg),$$
	$$P_{\scriptscriptstyle R}^{u}=\Bigg(1,-\frac{a_{\scriptscriptstyle R}^2-b_{\scriptscriptstyle R} \beta_{\scriptscriptstyle R}-\omega_{\scriptscriptstyle RS}^2}{b_{\scriptscriptstyle R}\omega_{\scriptscriptstyle RS}}\Bigg)\quad\text{and}\quad
	P_{\scriptscriptstyle R}^{s}=\Bigg(1,\frac{a_{\scriptscriptstyle R}^2-b_{\scriptscriptstyle R} \beta_{\scriptscriptstyle R}-\omega_{\scriptscriptstyle RS}^2}{b_{\scriptscriptstyle R}\omega_{\scriptscriptstyle RS}}\Bigg),$$
	where $\omega_{\scriptscriptstyle LS}=\sqrt{a^2_{\scriptscriptstyle L} + b_{\scriptscriptstyle L}c_{\scriptscriptstyle L}}$ and $\omega_{\scriptscriptstyle RS}=\sqrt{a^2_{\scriptscriptstyle R} + b_{\scriptscriptstyle R}c_{\scriptscriptstyle R}}$. Note that we have a symmetry between the points $P_{\scriptscriptstyle L}^{u}$ and $P_{\scriptscriptstyle L}^{s}$ (resp. $P_{\scriptscriptstyle R}^{u}$ and $P_{\scriptscriptstyle R}^{s}$) with respect to $x$-axis.  
	Define by $\tau$ the smallest ordinate value between the points $P_{\scriptscriptstyle R}^{s}$ and $P_{\scriptscriptstyle L}^{u}$, i.e. $\tau= \min\{(a_{\scriptscriptstyle R}^2-b_{\scriptscriptstyle R}\beta_{\scriptscriptstyle R}-\omega_{\scriptscriptstyle RS}^2)/b_{\scriptscriptstyle R}\omega_{\scriptscriptstyle RS},(a_{\scriptscriptstyle L}^2+b_{\scriptscriptstyle L}\beta_{\scriptscriptstyle L}-\omega_{\scriptscriptstyle LS}^2)/b_{\scriptscriptstyle L}\omega_{\scriptscriptstyle LS}\}$.
	Then, less than one reflection around the $y$-axis, we can assuming that $\tau=(a_{\scriptscriptstyle R}^2-b_{\scriptscriptstyle R}\beta_{\scriptscriptstyle R}-\omega_{\scriptscriptstyle RS}^2)/b_{\scriptscriptstyle R}\omega_{\scriptscriptstyle RS}$.
	
	As the vector field $X_{\scriptscriptstyle C}$ associated with the central subsystem from $\eqref{eq:01}|_{\epsilon=0}$ is $X_{\scriptscriptstyle C}(x,y)=(y,-x)$, if system $\eqref{eq:01}|_{\epsilon=0}$ is of type SCS and the ordinates of the points $P_{\scriptscriptstyle R}^{s}$ and $P_{\scriptscriptstyle L}^{u}$ are distinct (see Fig. \ref{fig:03} (a)) or if system $\eqref{eq:01}|_{\epsilon=0}$ is of type CCS (see Fig. \ref{fig:03} (c) or (d)), then we have a homoclinic loop passing through the points $P_{\scriptscriptstyle R}^{s}$ and $P_{\scriptscriptstyle R}^{u}$. Otherwise, if system $\eqref{eq:01}|_{\epsilon=0}$ is of type SCS and the ordinates of points $P_{\scriptscriptstyle R}^{s}$ and $P_{\scriptscriptstyle L}^{u}$ are the same (see Fig. \ref{fig:03} (b)) then we have a hetoclinic orbit passing through the points $P_{\scriptscriptstyle R}^{s}$, $P_{\scriptscriptstyle R}^{u}$, $P_{\scriptscriptstyle L}^{s}$ and $P_{\scriptscriptstyle L}^{u}$ .
	Moreover, the central subsystem from $\eqref{eq:01}|_{\epsilon=0}$ has a periodic orbit tangent to straight lines $x=\pm 1$ in the points $P_{\scriptscriptstyle R}=(1,0)$ and $P_{\scriptscriptstyle L}=(-1,0)$. 
	The Fig. \ref{fig:03}  shows the possibles phase portraits of the system $\eqref{eq:01}|_{\epsilon=0}$ of type SCS and CCS.
	
	\begin{figure}[h]
		\begin{center}		
			\begin{overpic}[width=4.8in]{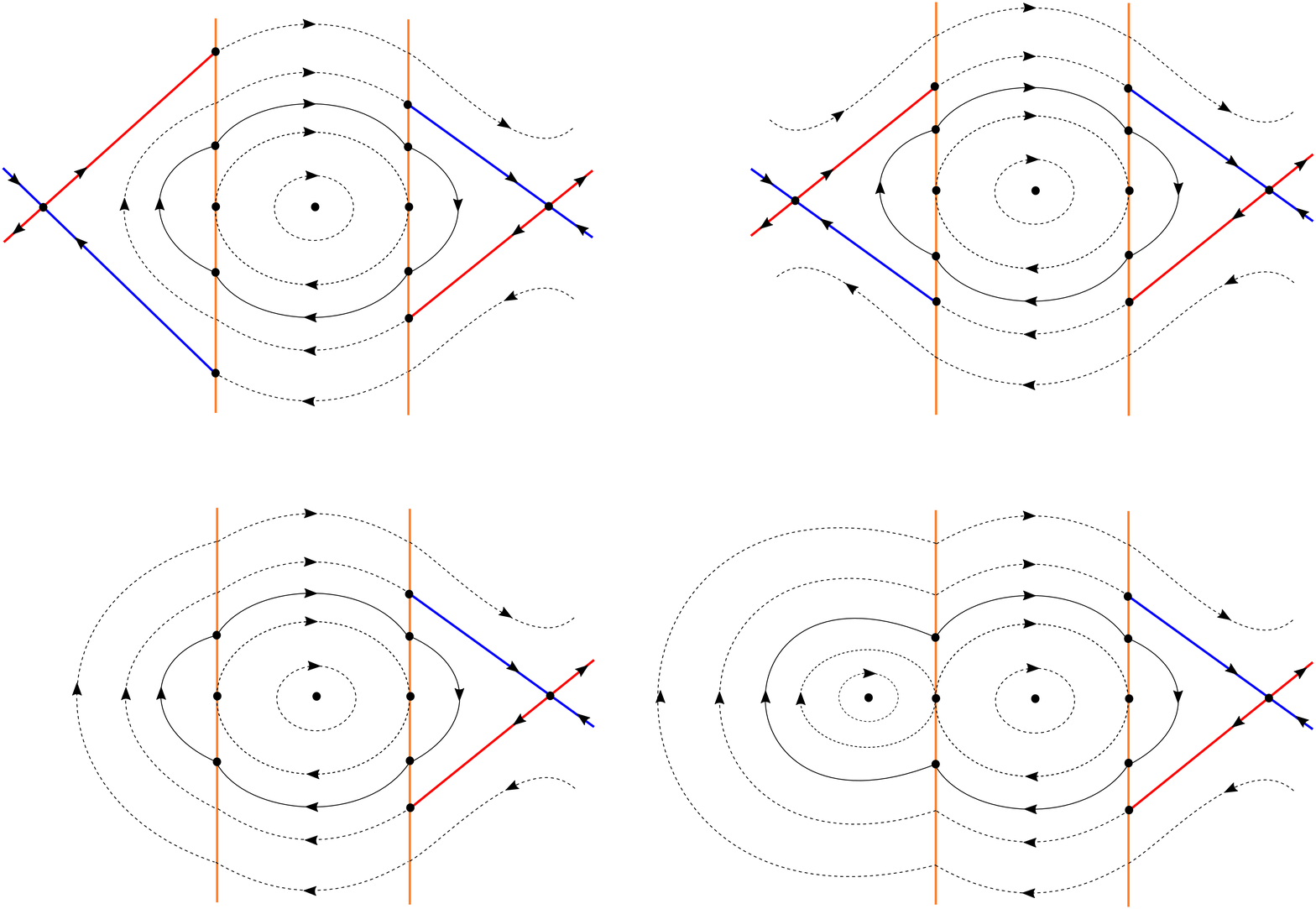}
				\put(30,36) {$\scriptstyle x=1$}
				\put(14,36) {$\scriptstyle x=-1$}
				\put(31.5,61.5) {$\scriptstyle P_{\scriptscriptstyle R}^{s}$}	
				\put(31.5,43.5) {$\scriptstyle P_{\scriptscriptstyle R}^{u}$}
				\put(13,66) {$\scriptstyle P_{\scriptscriptstyle L}^{u}$}	
				\put(13,39) {$\scriptstyle P_{\scriptscriptstyle L}^{s}$}				
				\put(31.5,58) {$\scriptstyle A$}
				\put(31.3,47) {$\scriptstyle A_1$}
				\put(13.5,47) {$\scriptstyle A_2$}
				\put(13.5,58.5) {$\scriptstyle A_3$}
				\put(31.5,52.5) {$\scriptstyle P_{\scriptscriptstyle R}$}
				\put(13,52.5) {$\scriptstyle P_{\scriptscriptstyle L}$}
				\put(22,34) {$(a)$}

				\put(84.5,36) {$\scriptstyle x=1$}
				\put(69,36) {$\scriptstyle x=-1$}
				\put(86.5,62.5) {$\scriptstyle P_{\scriptscriptstyle R}^{s}$}	
				\put(86.5,45) {$\scriptstyle P_{\scriptscriptstyle R}^{u}$}
				\put(68,63) {$\scriptstyle P_{\scriptscriptstyle L}^{u}$}	
				\put(68,44.5) {$\scriptstyle P_{\scriptscriptstyle L}^{s}$}				
				\put(86.5,59) {$\scriptstyle A$}
				\put(86.5,48.5) {$\scriptstyle A_1$}
				\put(68,48.5) {$\scriptstyle A_2$}
				\put(68,59.5) {$\scriptstyle A_3$}
				\put(86.5,54) {$\scriptstyle P_{\scriptscriptstyle R}$}
				\put(68,54) {$\scriptstyle P_{\scriptscriptstyle L}$}
				\put(77,34) {$(b)$}
				
				\put(30,-2) {$\scriptstyle x=1$}
				\put(14,-2) {$\scriptstyle x=-1$}
				\put(31.5,24) {$\scriptstyle P_{\scriptscriptstyle R}^{s}$}	
				\put(31.5,6.2) {$\scriptstyle P_{\scriptscriptstyle R}^{u}$}			
				\put(31.5,20.5) {$\scriptstyle A$}
				\put(31.3,10) {$\scriptstyle A_1$}
				\put(13.5,9.5) {$\scriptstyle A_2$}
				\put(13.4,21) {$\scriptstyle A_3$}
				\put(31.5,15.5) {$\scriptstyle P_{\scriptscriptstyle R}$}
				\put(13,15.5) {$\scriptstyle P_{\scriptscriptstyle L}$}
				\put(22,-4) {$(c)$}	
				
				\put(84.5,-2) {$\scriptstyle x=1$}
				\put(69,-2) {$\scriptstyle x=-1$}
				\put(86.5,24) {$\scriptstyle P_{\scriptscriptstyle R}^{s}$}	
				\put(86.5,6.2) {$\scriptstyle P_{\scriptscriptstyle R}^{u}$}			
				\put(86.5,20.5) {$\scriptstyle A$}
				\put(86.5,10) {$\scriptstyle A_1$}
				\put(68,8.5) {$\scriptstyle A_2$}
				\put(68,21.7) {$\scriptstyle A_3$}
				\put(86.4,15.5) {$\scriptstyle P_{\scriptscriptstyle R}$}
				\put(68,15.5) {$\scriptstyle P_{\scriptscriptstyle L}$}	
				\put(77,-4) {$(d)$}	
			\end{overpic}
		\end{center}
		\vspace{0.7cm}
		\caption{Phase portrait of system $\eqref{eq:01}|_{\epsilon=0}$ of type : (a) SCS with the ordinates of points $P_{\scriptscriptstyle R}^{s}$ and $P_{\scriptscriptstyle L}^{u}$ distinct; (b) SCS with the ordinates of points $P_{\scriptscriptstyle R}^{s}$ and $P_{\scriptscriptstyle L}^{u}$ equal; (c) CCS when left subsystem has a virtual center; (d) CCS when left subsystem has a real center.}\label{fig:03} 
	\end{figure} 
	Consider a initial point of form $A(h)=(1,h)$, with $h\in (0,\tau)$. By the hypothesis (H3), the system $\eqref{eq:01}|_{\epsilon=0}$ has a family of crossing periodic orbits that intersects the straight lines $x=\pm1$ at four points, $A(h)$, $A_1(h)=(1,a_1(h))$, with $a_1(h)<h$, and $A_2(h)=(-1,a_2(h))$, $A_3(h)=(-1,a_3(h))$, with $a_2(h)<a_3(h)$ satisfying
	\begin{equation*}
		\begin{aligned}
			& H^{\scriptscriptstyle R}(A(h))=H^{\scriptscriptstyle R}(A_1(h)), \\
			& H^{\scriptscriptstyle C}(A_1(h))=H^{\scriptscriptstyle C}(A_2(h)), \\
			& H^{\scriptscriptstyle L}(A_2(h))=H^{\scriptscriptstyle L}(A_3(h)), \\
			& H^{\scriptscriptstyle C}(A_3(h))=H^{\scriptscriptstyle C}(A(h)), 
		\end{aligned}
	\end{equation*}
	where $H^{\scriptscriptstyle R}$,  $H^{\scriptscriptstyle C}$ and  $H^{\scriptscriptstyle L}$ are given by \eqref{eq:fnh}. More precisely, we have the equations
	\begin{equation*}
		\begin{aligned}
			& \frac{b_{\scriptscriptstyle R}}{2}(h-a_1(h))(h+a_1(h))=0, \\
			& \frac{1}{2}(a_1(h)-a_2(h))(a_1(h)+a_2(h))=0, \\
			& \frac{b_{\scriptscriptstyle L}}{2}(a_2(h)-a_3(h))(a_2(h)+a_3(h))=0, \\
			& \frac{1}{2}(a_3(h)-h)(a_3(h)+h)=0. 
		\end{aligned}
	\end{equation*}
	As $a_1(h)<h$, $a_2(h)<a_3(h)$, $b_{\scriptscriptstyle R}>0$ and $b_{\scriptscriptstyle L}>0$, the only solution of system above is $a_1(h)=-h$, $a_2(h)=-h$ and $a_3(h)=h$, i.e. we have the four points given by $A(h)=(1,h)$, $A_1(h)=(1,-h)$, $A_2(h)=(-1,-h)$ and $A_3(h)=(-1,h)$. Moreover, system  $\eqref{eq:01}|_{\epsilon=0}$ has a periodic orbit $L_h$ passing through these points, for all $h\in(0,\tau)$. If $h\in[\tau,\infty)$ then the orbit of the system $\eqref{eq:01}|_{\epsilon=0}$ with initial condition in $A(h)$ do not return to straight line $x=1$ to positive times, i.e. the system $\eqref{eq:01}|_{\epsilon=0}$ has no periodic orbit passing thought the point $A(h)$. Therefore, if $h\in(0,\tau)$ the system $\eqref{eq:01}|_{\epsilon=0}$ has a periodic annulus, formed by the periodic orbits $L_h$, limited by one (see Fig. \ref{fig:03} (a)--(c)) or two (see Fig. \ref{fig:03} (d)) periodic orbits tangent to the straight lines $x=\pm1$,  when $h=0$, and a homoclinic loop (see Fig. \ref{fig:03} (a), (c) and (d)) or heteroclinic orbit (see Fig. \ref{fig:03} (b)), when $h=\tau$. Therefore, item (a) is proven.
	
	
	To prove item (b), suppose that the system $\eqref{eq:01}|_{\epsilon=0}$ is of type CCC. The central subsystem from $\eqref{eq:01}|_{\epsilon=0}$ has a periodic orbit tangent to straight lines $x=\pm 1$ in the points $P_{\scriptscriptstyle R}=(1,0)$ and $P_{\scriptscriptstyle L}=(-1,0)$. Moreover, as in the previous case, for each $h\in(0,\infty)$ we have a periodic orbit $L_h$ passing through points $A(h)=(1,h)$, $A_1(h)=(1,-h)$, $A_2(h)=(-1,-h)$ and $A_3(h)=(-1,h)$. Therefore, the system $\eqref{eq:01}|_{\epsilon=0}$ has a continuum of periodic orbit formed by the periodic orbits $L_h$, with $h\in(0,\infty)$, and limited by one (see Fig. \ref{fig:04} (a)), two (see Fig. \ref{fig:04} (a)) or three (see Fig. \ref{fig:04} (c)) periodic orbits tangent to straight lines $x=\pm1$,  when $h=0$. The Fig. \ref{fig:04}  shows the possibles phase portraits of the system $\eqref{eq:01}|_{\epsilon=0}$ of type CCC.
	\begin{figure}[h]
		\begin{center}		
			\begin{overpic}[width=4.8in]{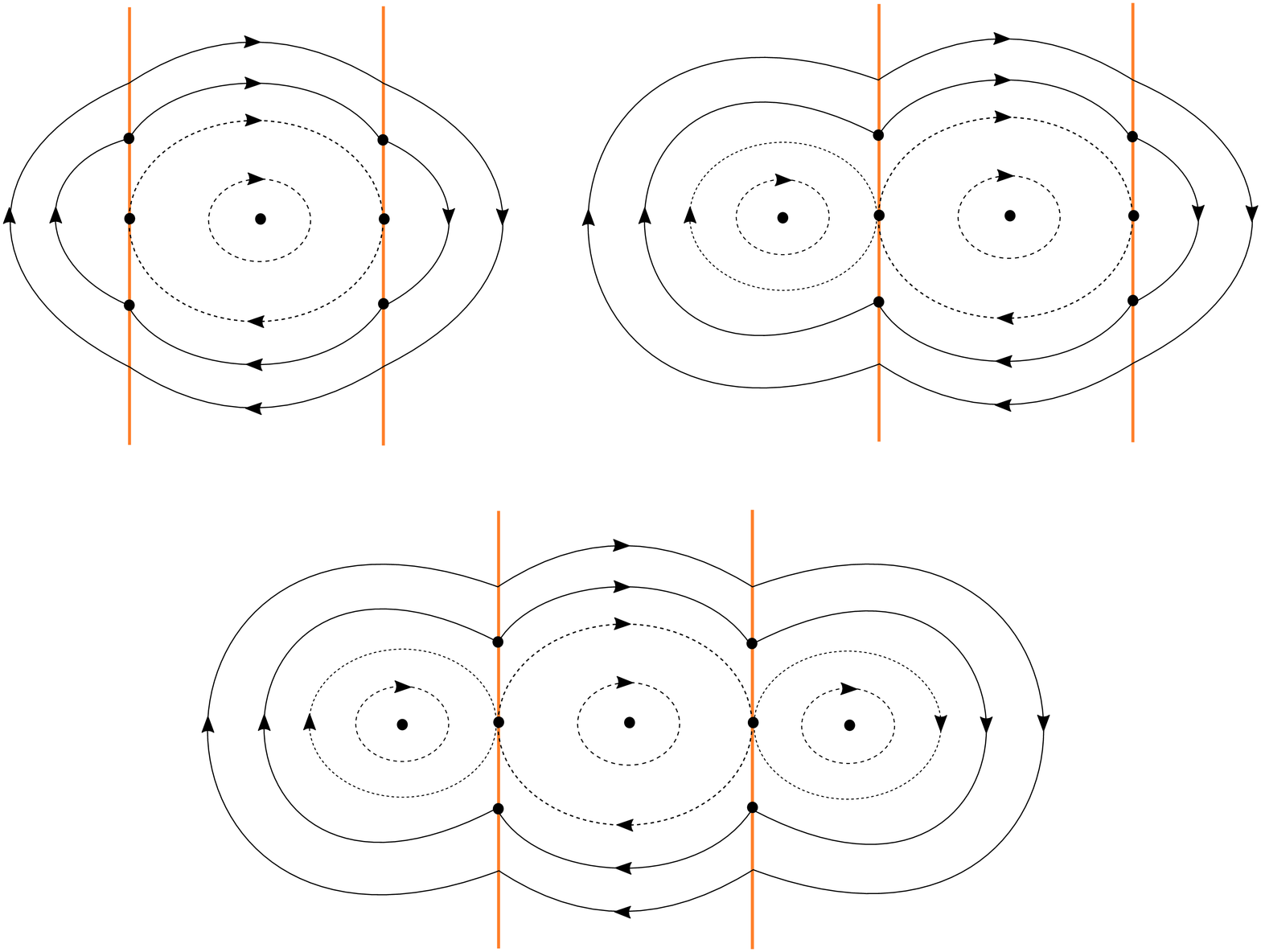}
				\put(28,38) {$ x=1$}
				\put(7,38) {$ x=-1$}			
				\put(31,65) {$ A$}
				\put(31,50) {$ A_1$}
				\put(6,50) {$ A_2$}
				\put(6,65) {$ A_3$}
				\put(31,58) {$ P_{\scriptscriptstyle R}$}
				\put(6,58) {$ P_{\scriptscriptstyle L}$}
				\put(19,37) {$(a)$}
				
				\put(87,38) {$ x=1$}
				\put(67,38) {$ x=-1$}			
				\put(91,65) {$ A$}
				\put(91,50) {$ A_1$}
				\put(66,48.5) {$ A_2$}
				\put(66,66.5) {$ A_3$}
				\put(91,58) {$ P_{\scriptscriptstyle R}$}
				\put(66,58) {$ P_{\scriptscriptstyle L}$}
				\put(79,37) {$(b)$}
				
				\put(57,-2) {$ x=1$}
				\put(36.5,-2) {$ x=-1$}			
				\put(60,25.7) {$ A$}
				\put(60,8) {$ A_1$}
				\put(35.5,8) {$ A_2$}
				\put(35.5,26.3) {$ A_3$}
				\put(60.5,18) {$ P_{\scriptscriptstyle R}$}
				\put(35.5,18) {$ P_{\scriptscriptstyle L}$}
				\put(48,-3) {$(c)$}
			\end{overpic}
		\end{center}
		\vspace{0.7cm}
		\caption{Phase portrait of system $\eqref{eq:01}|_{\epsilon=0}$ of type CCC when:  (a) the left and right subsystems have virtual centers; (b) the left subsystem has a real center and right subsystem has a virtual center; (c) the left and right subsystems have real centers.}\label{fig:04}
	\end{figure} 
\end{proof}

\bigskip

The coefficients that multiply the integrals of first order Melnikov function \eqref{eq:mel} associated to system \eqref{eq:01} can be easily calculated. More precisely, we have the immediate corollary.  

\begin{corollary}
	Let $J$ be the interval of definition of Melnikov function \eqref{eq:mel}. For $h\in J$,
	$$
	\frac{H_y^{\scriptscriptstyle R}(A)}{H_y^{\scriptscriptstyle C}(A)}=b_{\scriptscriptstyle R},\quad \frac{H_y^{\scriptscriptstyle R}(A)H_y^{\scriptscriptstyle C}(A_3)}{H_y^{\scriptscriptstyle C}(A)H_y^{\scriptscriptstyle L}(A_3)}=\frac{b_{\scriptscriptstyle R}}{b_{\scriptscriptstyle L}},\quad
	\frac{H_y^{\scriptscriptstyle R}(A)H_y^{\scriptscriptstyle C}(A_3)H_y^{\scriptscriptstyle L}(A_2)}{H_y^{\scriptscriptstyle C}(A)H_y^{\scriptscriptstyle L}(A_3)H_y^{\scriptscriptstyle C}(A_2)}=b_{\scriptscriptstyle R}$$
	and
	$$\frac{H_y^{\scriptscriptstyle R}(A)H_y^{\scriptscriptstyle C}(A_3)H_y^{\scriptscriptstyle L}(A_2)H_y^{\scriptscriptstyle C}(A_1)}{H_y^{\scriptscriptstyle C}(A)H_y^{\scriptscriptstyle L}(A_3)H_y^{\scriptscriptstyle C}(A_2)H_y^{\scriptscriptstyle R}(A_1)}=1.
	$$
	Then, the first order Melnikov function associated to system \eqref{eq:01} can be written as
	\begin{equation}\label{eq:mel01}
		\begin{aligned}
		M(h) = &\,\, b_{\scriptscriptstyle R}\int_{\widehat{A_3A}}g_{\scriptscriptstyle C}dx-f_{\scriptscriptstyle C}dy+\frac{b_{\scriptscriptstyle R}}{b_{\scriptscriptstyle L}} \int_{\widehat{A_2A_3}}g_{\scriptscriptstyle L}dx-f_{\scriptscriptstyle L}dy \\
		      &\, + b_{\scriptscriptstyle R}\int_{\widehat{A_1A_2}}g_{\scriptscriptstyle C}dx-f_{\scriptscriptstyle C}dy+\int_{\widehat{AA_1}}g_{\scriptscriptstyle R}dx-f_{\scriptscriptstyle R}dy.
		\end{aligned}
	\end{equation}
\end{corollary}

\bigskip

In what follows, we will determinate the first order Melnikov function associated to system \eqref{eq:01} when the system $\eqref{eq:01}|_{\epsilon=0}$ is of the type SCS, CCS and CCC. For this, we define the functions:

\begin{equation}\label{eq:func}
	\begin{aligned}
		f_0(h) = \hspace{0.1cm} & h, \quad h\in(0,\infty), \\
		f_{\scriptscriptstyle C}^{\scriptscriptstyle C}(h) =\hspace{0.1cm} & (h^2+1)\arccos\bigg(\frac{h^2-1}{h^2+1}\bigg),\quad h\in(0,\infty), \\
		f_{\scriptscriptstyle R}^{\scriptscriptstyle C}(h) =\hspace{0.1cm} & ((a_{\scriptscriptstyle R}^2 - b_{\scriptscriptstyle R} \beta_{\scriptscriptstyle R})^2 + (2 a_{\scriptscriptstyle R}^2 + 
		b_{\scriptscriptstyle R}^2 h^2 - 2 b_{\scriptscriptstyle R} \beta_{\scriptscriptstyle R}) \omega_{\scriptscriptstyle RC}^2 ) F_{\scriptscriptstyle R}^{\scriptscriptstyle C}(h)\\
		&+ \omega_{\scriptscriptstyle RC}^4 F_{\scriptscriptstyle R}^{\scriptscriptstyle C}(h), \quad h\in(0,\infty), \\
		f_{\scriptscriptstyle L}^{\scriptscriptstyle C}(h) =\hspace{0.1cm} &  ((a_{\scriptscriptstyle L}^2 + b_{\scriptscriptstyle L} \beta_{\scriptscriptstyle L})^2 + (2 a_{\scriptscriptstyle L}^2 + 
		b_{\scriptscriptstyle L}^2 h^2 + 2 b_{\scriptscriptstyle L}\beta_{\scriptscriptstyle L}) \omega_{\scriptscriptstyle LC}^2 ) F_{\scriptscriptstyle L}^{\scriptscriptstyle C}(h)\\
		& + \omega_{\scriptscriptstyle LC}^4 F_{\scriptscriptstyle L}^{\scriptscriptstyle C}(h), \quad h\in(0,\infty),\\
		f_{\scriptscriptstyle R}^{\scriptscriptstyle S}(h) =\hspace{0.1cm} & (a_{\scriptscriptstyle R}^2 - b_{\scriptscriptstyle R} \beta_{\scriptscriptstyle R} + 
		b_{\scriptscriptstyle R} \omega_{\scriptscriptstyle RS} h - \omega_{\scriptscriptstyle RS}^2) \\
		& \times(-a_{\scriptscriptstyle R}^2 + 
		b_{\scriptscriptstyle R} \beta_{\scriptscriptstyle R} + b_{\scriptscriptstyle R} \omega_{\scriptscriptstyle RS} h + \omega_{\scriptscriptstyle RS}^2)F_{\scriptscriptstyle R}^{\scriptscriptstyle S}(h), \quad h\in(0,\tau), \\
		f_{\scriptscriptstyle L}^{\scriptscriptstyle S}(h) =\hspace{0.1cm} & (-a_{\scriptscriptstyle L}^2 - b_{\scriptscriptstyle L} \beta_{\scriptscriptstyle L} + 
		b_{\scriptscriptstyle L} \omega_{\scriptscriptstyle LS} h + \omega_{\scriptscriptstyle LS}^2)\\
		& \times (a_{\scriptscriptstyle L}^2 + 
		b_{\scriptscriptstyle L} \beta_{\scriptscriptstyle L} + 	b_{\scriptscriptstyle L} \omega_{\scriptscriptstyle LS} h - \omega_{\scriptscriptstyle LS}^2 ) F_{\scriptscriptstyle L}^{\scriptscriptstyle S}(h), \quad h\in(0,\tau),\\	
	\end{aligned}
\end{equation}	
with
\begin{equation*}
	\begin{aligned}
		F_{\scriptscriptstyle R}^{\scriptscriptstyle C}(h) =& \arccos\bigg(1-\frac{2 b_{\scriptscriptstyle R}^2  \omega_{\scriptscriptstyle RC}^2 h^2}{(a_{\scriptscriptstyle R}^2 - b_{\scriptscriptstyle R} \beta_{\scriptscriptstyle R})^2 + (2 a_{\scriptscriptstyle R}^2 + b_{\scriptscriptstyle R}^2 h^2 - 
			2 b_{\scriptscriptstyle R} \beta_{\scriptscriptstyle R}) \omega_{\scriptscriptstyle RC}^2 + \omega_{\scriptscriptstyle RC}^4}\bigg), \\
		F_{\scriptscriptstyle L}^{\scriptscriptstyle C}(h) =& \arccos\bigg(1-\frac{2 b_{\scriptscriptstyle L}^2 \omega_{\scriptscriptstyle LC}^2  h^2 }{(a_{\scriptscriptstyle L}^2 + b_{\scriptscriptstyle L} \beta_{\scriptscriptstyle L})^2 + (2 a_{\scriptscriptstyle L}^2 + b_{\scriptscriptstyle L}^2 h^2 + 
			2 b_{\scriptscriptstyle L} \beta_{\scriptscriptstyle L}) \omega_{\scriptscriptstyle LC}^2 + \omega_{\scriptscriptstyle LC}^4}\bigg), \\
		F_{\scriptscriptstyle R}^{\scriptscriptstyle S}(h) =& \log\bigg(1-\frac{2 b_{\scriptscriptstyle R} \omega_{\scriptscriptstyle RS} h}{-a_{\scriptscriptstyle R}^2 + b_{\scriptscriptstyle R} \beta_{\scriptscriptstyle R} + b_{\scriptscriptstyle R} \omega_{\scriptscriptstyle RS} h + \omega_{\scriptscriptstyle RS}^2}\bigg), \\
		F_{\scriptscriptstyle L}^{\scriptscriptstyle S}(h) =& \log\bigg(1+\frac{2 b_{\scriptscriptstyle L} \omega_{\scriptscriptstyle LS} h}{a_{\scriptscriptstyle L}^2 + b_{\scriptscriptstyle L} \beta_{\scriptscriptstyle L} - b_{\scriptscriptstyle L} \omega_{\scriptscriptstyle LS} h - \omega_{\scriptscriptstyle LS}^2}\bigg), \\	
	\end{aligned}
\end{equation*}	
where $\omega_{i{\scriptscriptstyle S}}=\sqrt{a^2_i + b_ic_i}$ and $\omega_{i{\scriptscriptstyle C}}=\sqrt{-a^2_i - b_ic_i}$, for $i=L,R$.

\bigskip






\medskip

\begin{theorem}\label{theo:scs}
	Suppose that system $\eqref{eq:01}|_{\epsilon=0}$ is of the type SCS. Then the first order Melnikov function $M(h)$ associated with system \eqref{eq:01} can be expressed as	
	\begin{equation}\label{eq:melscs}
		M(h)=k_0f_0(h)+k_{\scriptscriptstyle C}^{\scriptscriptstyle C}f_{\scriptscriptstyle C}^{\scriptscriptstyle C}(h)+k_{\scriptscriptstyle R}^{\scriptscriptstyle S}f_{\scriptscriptstyle R}^{\scriptscriptstyle S}(h)+k_{\scriptscriptstyle L}^{\scriptscriptstyle S}f_{\scriptscriptstyle L}^{\scriptscriptstyle S}(h),
	\end{equation}
	for $h\in(0,\tau)$, where the functions $f_0,f_{\scriptscriptstyle C}^{\scriptscriptstyle C},f_{\scriptscriptstyle R}^{\scriptscriptstyle S},f_{\scriptscriptstyle L}^{\scriptscriptstyle S}$ are the ones defined in \eqref{eq:func}. Here the coefficients $k_0$ and $k_i^j$, for $i=L,C,R$ and $j=C,S$, depend on the parameters of system \eqref{eq:01}.
\end{theorem}
\begin{proof}
	The orbit $(x_{\scriptscriptstyle R}(x,y),y_{\scriptscriptstyle R}(x,y))$ of the system $\eqref{eq:01}|_{\epsilon=0}$, such that $(x_{\scriptscriptstyle R}(0,0),y_{\scriptscriptstyle R}(0,0))=(1,h)$, is given by
	\begin{equation*}
		\begin{aligned}
			x_{\scriptscriptstyle R}(t)=\,&\frac{e^{-t \omega_{\scriptscriptstyle RS}}}{2\omega_{\scriptscriptstyle RS}^2}(b_{\scriptscriptstyle R}\beta_{\scriptscriptstyle R}-a_{\scriptscriptstyle R}^2+\omega_{\scriptscriptstyle RS}^2-b_{\scriptscriptstyle R}\omega_{\scriptscriptstyle RS}h)+\frac{1}{2\omega_{\scriptscriptstyle RS}^2}(2a_{\scriptscriptstyle R}^2-2b_{\scriptscriptstyle R}\beta_{\scriptscriptstyle R})\\
			& +\frac{e^{t \omega_{\scriptscriptstyle RS}}}{2\omega_{\scriptscriptstyle RS}^2}(b_{\scriptscriptstyle R}\beta_{\scriptscriptstyle R}-a_{\scriptscriptstyle R}^2+b_{\scriptscriptstyle R}\omega_{\scriptscriptstyle RS}h+\omega_{\scriptscriptstyle RS}^2), \\
			y_{\scriptscriptstyle R}(t)=\,&\frac{e^{-t \omega_{\scriptscriptstyle RS}}}{2b_{\scriptscriptstyle R}\omega_{\scriptscriptstyle RS}^2}(a_{\scriptscriptstyle R}^3-a_{\scriptscriptstyle R}b_{\scriptscriptstyle R}\beta_{\scriptscriptstyle R}+a_{\scriptscriptstyle R}^2\omega_{\scriptscriptstyle RS}+a_{\scriptscriptstyle R}b_{\scriptscriptstyle R}\omega_{\scriptscriptstyle RS}h-b_{\scriptscriptstyle R}\beta_{\scriptscriptstyle R}\omega_{\scriptscriptstyle RS}-a_{\scriptscriptstyle R}\omega_{\scriptscriptstyle RS}^2)\\
			& +\frac{e^{-t \omega_{\scriptscriptstyle RS}}}{2b_{\scriptscriptstyle R}\omega_{\scriptscriptstyle RS}^2}(b_{\scriptscriptstyle R}\omega_{\scriptscriptstyle RS}^2h-\omega_{\scriptscriptstyle RS}^3) +\frac{1}{2b_{\scriptscriptstyle R}\omega_{\scriptscriptstyle RS}^2}(-2a_{\scriptscriptstyle R}^3+2a_{\scriptscriptstyle R}b_{\scriptscriptstyle R}\beta_{\scriptscriptstyle R}+2a_{\scriptscriptstyle R}\omega_{\scriptscriptstyle RS}^2)\\
			& +\frac{e^{t \omega_{\scriptscriptstyle RS}}}{2b_{\scriptscriptstyle R}\omega_{\scriptscriptstyle RS}^2}(a_{\scriptscriptstyle R}^3-a_{\scriptscriptstyle R}b_{\scriptscriptstyle R}\beta_{\scriptscriptstyle R}-a_{\scriptscriptstyle R}^2\omega_{\scriptscriptstyle RS}-a_{\scriptscriptstyle R}b_{\scriptscriptstyle R}\omega_{\scriptscriptstyle RS}h+b_{\scriptscriptstyle R}\beta_{\scriptscriptstyle R}\omega_{\scriptscriptstyle RS})\\
			& +\frac{e^{t \omega_{\scriptscriptstyle RS}}}{2b_{\scriptscriptstyle R}\omega_{\scriptscriptstyle RS}^2}(-a_{\scriptscriptstyle R}\omega_{\scriptscriptstyle RS}^2+b_{\scriptscriptstyle R}\omega_{\scriptscriptstyle RS}^2h+\omega_{\scriptscriptstyle RS}^3).
		\end{aligned}
	\end{equation*}
	The fly time of the orbit $(x_{\scriptscriptstyle R}(x,y),y_{\scriptscriptstyle R}(x,y))$, from $A(h)=(1,h)$ to $A_1(h)=(1,-h)$, is 
	$$t_{\scriptscriptstyle R}=\frac{1}{\omega_{\scriptscriptstyle RS}}\log\Bigg(1-\frac{2b_{\scriptscriptstyle R}\omega_{\scriptscriptstyle RS}h}{-a_{\scriptscriptstyle R}^2+b_{\scriptscriptstyle R}\beta_{\scriptscriptstyle R}+b_{\scriptscriptstyle R}\omega_{\scriptscriptstyle RS}h+\omega_{\scriptscriptstyle RS}^2}\Bigg).$$
	Now, for $g_{\scriptscriptstyle R}$ and $f_{\scriptscriptstyle R}$ defined in \eqref{eq:02} and \eqref{eq:03}, respectively, we have
	\begin{equation}\label{sys:r}
		\begin{aligned}
			&\int_{\widehat{AA_1}}g_{\scriptscriptstyle R}dx-f_{\scriptscriptstyle R}dy=\\ &=\int_{0}^{t_{\scriptscriptstyle R}}g_{\scriptscriptstyle R}(x_{\scriptscriptstyle R}(t), y_{\scriptscriptstyle R}(t)) \dfrac{d}{dt}x_{\scriptscriptstyle R}(t) - 
			f_{\scriptscriptstyle R}(x_{\scriptscriptstyle R}(t), y_{\scriptscriptstyle R}(t)) \dfrac{d}{dt}y_{\scriptscriptstyle R}(t) \\
			& = \alpha_1h+\alpha_2f_{\scriptscriptstyle R}^{\scriptscriptstyle S}(h),		
		\end{aligned}
	\end{equation}
	with 
	$$
	\alpha_1 = p_{\scriptscriptstyle R}+2r_{\scriptscriptstyle R}-u_{\scriptscriptstyle R}+\frac{(p_{\scriptscriptstyle R}+u_{\scriptscriptstyle R})(a_{\scriptscriptstyle R}^2-b_{\scriptscriptstyle R}\beta_{\scriptscriptstyle R})}{\omega_{\scriptscriptstyle RS}^2}\quad\text{and}\quad
	\alpha_2  = \frac{p_{\scriptscriptstyle R}+u_{\scriptscriptstyle R}}{2b_{\scriptscriptstyle R}\omega_{\scriptscriptstyle RS}^3}.		
	$$
	The orbit $(x_{\scriptscriptstyle C1}(x,y),y_{\scriptscriptstyle C1}(x,y))$ of the system $\eqref{eq:01}|_{\epsilon=0}$, such that $(x_{\scriptscriptstyle C1}(0,0),$ $y_{\scriptscriptstyle C1}(0,0))=(1,-h)$, is given by
	\begin{equation*}
		\begin{aligned}
			x_{\scriptscriptstyle C1}(t)&=\cos(t) - h \sin(t), \\
			y_{\scriptscriptstyle C1}(t)&=-h \cos(t) - \sin(t).
		\end{aligned}
	\end{equation*}
	The fly time of the orbit $(x_{\scriptscriptstyle C1}(x,y),y_{\scriptscriptstyle C1}(x,y))$, from $A_1(h)=(1,-h)$ to $A_2(h)=(-1,-h)$, is
	$$t_{\scriptscriptstyle C1}=\arccos\Bigg(\frac{h^2-1}{h^2+1}\Bigg)$$
	Now, for $g_{\scriptscriptstyle C}$ and $f_{\scriptscriptstyle C}$ defined in \eqref{eq:02} and \eqref{eq:03}, respectively, we obtain 
	\begin{equation}\label{sys:c1}
		\begin{aligned}
			&\int_{\widehat{A_1A_2}}g_{\scriptscriptstyle C}dx-f_{\scriptscriptstyle C}dy=\\ 
			&=\int_{0}^{t_{\scriptscriptstyle C1}}g_{\scriptscriptstyle C}(x_{\scriptscriptstyle C1}(t), y_{\scriptscriptstyle C1}(t)) \dfrac{d}{dt}x_{\scriptscriptstyle C1}(t) - 
			f_{\scriptscriptstyle C}(x_{\scriptscriptstyle C1}(t), y_{\scriptscriptstyle C1}(t)) \dfrac{d}{dt}y_{\scriptscriptstyle C1}(t) \\
			& = -2v_{\scriptscriptstyle C}+\alpha_3h+\alpha_4f_{\scriptscriptstyle C}^{\scriptscriptstyle C}(h),		
		\end{aligned}
	\end{equation}
	with
	$$
	\alpha_3 = u_{\scriptscriptstyle C}-p_{\scriptscriptstyle C}\quad\text{and}\quad
	\alpha_4 = \frac{p_{\scriptscriptstyle C}+u_{\scriptscriptstyle C}}{2}.		
	$$
	The orbit $(x_{\scriptscriptstyle L}(x,y),y_{\scriptscriptstyle L}(x,y))$ of the system $\eqref{eq:01}|_{\epsilon=0}$, such that $(x_{\scriptscriptstyle L}(0,0),$ $y_{\scriptscriptstyle L}(0,0))=(-1,-h)$, is given by
	\begin{equation*}
		\begin{aligned}
			x_{\scriptscriptstyle L}(t)=&-\frac{e^{-t\omega_{\scriptscriptstyle LS}}}{2\omega_{\scriptscriptstyle LS}^2}(-a_{\scriptscriptstyle L}^2-b_{\scriptscriptstyle L}\beta_{\scriptscriptstyle L}-b_{\scriptscriptstyle L}\omega_{\scriptscriptstyle LS}h+\omega_{\scriptscriptstyle LS}^2)-\frac{1}{2\omega_{\scriptscriptstyle LS}^2}(2a_{\scriptscriptstyle L}^2+2b_{\scriptscriptstyle L}\beta_{\scriptscriptstyle L})\\
			&-\frac{e^{t\omega_{\scriptscriptstyle LS}}}{2\omega_{\scriptscriptstyle LS}^2}(-a_{\scriptscriptstyle L}^2-b_{\scriptscriptstyle L}\beta_{\scriptscriptstyle L}+b_{\scriptscriptstyle L}\omega_{\scriptscriptstyle LS}h+\omega_{\scriptscriptstyle LS}^2), \\
			y_{\scriptscriptstyle L}(t)=&-\frac{e^{-t\omega_{\scriptscriptstyle LS}}}{2b_{\scriptscriptstyle L}\omega_{\scriptscriptstyle LS}^2}(a_{\scriptscriptstyle L}^3+a_{\scriptscriptstyle L}b_{\scriptscriptstyle L}\beta_{\scriptscriptstyle L}+a_{\scriptscriptstyle L}^2\omega_{\scriptscriptstyle LS}+a_{\scriptscriptstyle L}b_{\scriptscriptstyle L}\omega_{\scriptscriptstyle LS}h+b_{\scriptscriptstyle L}\beta_{\scriptscriptstyle L}\omega_{\scriptscriptstyle LS}-a_{\scriptscriptstyle L}\omega_{\scriptscriptstyle LS}^2) \\
			&-\frac{e^{-t\omega_{\scriptscriptstyle LS}}}{2b_{\scriptscriptstyle L}\omega_{\scriptscriptstyle LS}^2}(b_{\scriptscriptstyle L}\omega_{\scriptscriptstyle LS}^2h-\omega_{\scriptscriptstyle LS}^3) -\frac{1}{2b_{\scriptscriptstyle L}\omega_{\scriptscriptstyle LS}^2}(-2a_{\scriptscriptstyle L}^3-2a_{\scriptscriptstyle L}b_{\scriptscriptstyle L}\beta_{\scriptscriptstyle L}+2a_{\scriptscriptstyle L}\omega_{\scriptscriptstyle LS}^2)\\
			& -\frac{e^{t\omega_{\scriptscriptstyle LS}}}{2b_{\scriptscriptstyle L}\omega_{\scriptscriptstyle LS}^2}(a_{\scriptscriptstyle L}^3+a_{\scriptscriptstyle L}b_{\scriptscriptstyle L}\beta_{\scriptscriptstyle L}-a_{\scriptscriptstyle L}^2\omega_{\scriptscriptstyle LS}-a_{\scriptscriptstyle L}b_{\scriptscriptstyle L}\omega_{\scriptscriptstyle LS}h-b_{\scriptscriptstyle L}\beta_{\scriptscriptstyle L}\omega_{\scriptscriptstyle LS}) \\
			&-\frac{e^{t\omega_{\scriptscriptstyle LS}}}{2b_{\scriptscriptstyle L}\omega_{\scriptscriptstyle LS}^2}(-a_{\scriptscriptstyle L}\omega_{\scriptscriptstyle LS}^2+b_{\scriptscriptstyle L}\omega_{\scriptscriptstyle LS}^2h+\omega_{\scriptscriptstyle LS}^3).
		\end{aligned}
	\end{equation*}
	The fly time of the orbit $(x_{\scriptscriptstyle L}(x,y),y_{\scriptscriptstyle L}(x,y))$, from $A_2(h)=(-1,-h)$ to $A_3(h)=(-1,h)$, is 
	$$t_{\scriptscriptstyle L}=\frac{1}{\omega_{\scriptscriptstyle LS}}\log\Bigg(1+\frac{2b_{\scriptscriptstyle L}\omega_{\scriptscriptstyle LS}h}{a_{\scriptscriptstyle L}^2+b_{\scriptscriptstyle L}\beta_{\scriptscriptstyle L}-b_{\scriptscriptstyle L}\omega_{\scriptscriptstyle LS}h-\omega_{\scriptscriptstyle LS}^2}\Bigg).$$
	Now, for $g_{\scriptscriptstyle L}$ and $f_{\scriptscriptstyle L}$ defined in \eqref{eq:02} and \eqref{eq:03}, respectively, we obtain 
	\begin{equation}\label{sys:l}
		\begin{aligned}
			&\int_{\widehat{A_2A_3}}g_{\scriptscriptstyle L}dx-f_{\scriptscriptstyle L}dy=\\ 
			&=\int_{0}^{t_{\scriptscriptstyle L}}g_{\scriptscriptstyle L}(x_{\scriptscriptstyle L}(t), y_{\scriptscriptstyle L}(t)) \dfrac{d}{dt}x_{\scriptscriptstyle L}(t) - 
			f_{\scriptscriptstyle L}(x_{\scriptscriptstyle L}(t), y_{\scriptscriptstyle L}(t)) \dfrac{d}{dt}y_{\scriptscriptstyle L}(t) \\
			& =\alpha_5h+\alpha_6f_{\scriptscriptstyle L}^{\scriptscriptstyle S}(h) ,		
		\end{aligned}
	\end{equation}	
	with
	$$
	\alpha_5 = p_{\scriptscriptstyle L}-2r_{\scriptscriptstyle L}-u_{\scriptscriptstyle L}+\frac{(p_{\scriptscriptstyle L}+u_{\scriptscriptstyle L})(a_{\scriptscriptstyle L}^2+b_{\scriptscriptstyle L}\beta_{\scriptscriptstyle L})}{\omega_{\scriptscriptstyle LS}^2}\quad\text{and}\quad
	\alpha_6  = \frac{p_{\scriptscriptstyle L}+u_{\scriptscriptstyle L}}{2b_{\scriptscriptstyle L}\omega_{\scriptscriptstyle LS}^3}.
	$$
	Finally, the orbit $(x_{\scriptscriptstyle C2}(x,y),y_{\scriptscriptstyle C2}(x,y))$ of the system $\eqref{eq:01}|_{\epsilon=0}$, such that $(x_{\scriptscriptstyle C2}(0,0),y_{\scriptscriptstyle C2}(0,0))=(-1,h)$, is given by
	\begin{equation*}
		\begin{aligned}
			x_{\scriptscriptstyle C2}(t)&=-\cos(t) + h \sin(t), \\
			y_{\scriptscriptstyle C2}(t)&=h \cos(t) + \sin(t).
		\end{aligned}
	\end{equation*}
	The fly time of the orbit $(x_{\scriptscriptstyle C2}(x,y),y_{\scriptscriptstyle C2}(x,y))$, from $A_3(h)=(-1,h)$ to $A(h)=(1,h)$, is
	$$t_{\scriptscriptstyle C2}=\arccos\Bigg(\frac{h^2-1}{h^2+1}\Bigg)$$
	Now, for $g_{\scriptscriptstyle C}$ and $f_{\scriptscriptstyle C}$ defined in \eqref{eq:02} and \eqref{eq:03}, respectively, we obtain 
	\begin{equation}\label{sys:c2}
		\begin{aligned}
			&\int_{\widehat{A_3A}}g_{\scriptscriptstyle C}dx-f_{\scriptscriptstyle C}dy=\\
			 & =\int_{0}^{t_{\scriptscriptstyle C2}}g_{\scriptscriptstyle C}(x_{\scriptscriptstyle C2}(t), y_{\scriptscriptstyle C2}(t)) \dfrac{d}{dt}x_{\scriptscriptstyle C2}(t) - 
			f_{\scriptscriptstyle C}(x_{\scriptscriptstyle C2}(t), y_{\scriptscriptstyle C2}(t)) \dfrac{d}{dt}y_{\scriptscriptstyle C2}(t) \\
			& = 2v_{\scriptscriptstyle C}+\alpha_3h+\alpha_4f_{\scriptscriptstyle C}^{\scriptscriptstyle C}(h).		
		\end{aligned}
	\end{equation}
	Therefore, replacing \eqref{sys:r}, \eqref{sys:c1}, \eqref{sys:l} and \eqref{sys:c2} in \eqref{eq:mel01}, we obtain
	\begin{equation*}
		M(h)=k_0f_0(h)+k_{\scriptscriptstyle C}^{\scriptscriptstyle C}f_{\scriptscriptstyle C}^{\scriptscriptstyle C}(h)+k_{\scriptscriptstyle R}^{\scriptscriptstyle S}f_{\scriptscriptstyle R}^{\scriptscriptstyle S}(h)+k_{\scriptscriptstyle L}^{\scriptscriptstyle S}f_{\scriptscriptstyle L}^{\scriptscriptstyle S}(h)
	\end{equation*} 
	with
	$$k_0=\alpha_1+2b_{\scriptscriptstyle R}\alpha_3+\frac{b_{\scriptscriptstyle R}}{b_{\scriptscriptstyle L}}\alpha_5,\quad k_{\scriptscriptstyle C}^{\scriptscriptstyle C}=2b_{\scriptscriptstyle R}\alpha_4,\quad k_{\scriptscriptstyle R}^{\scriptscriptstyle S}=\alpha_2\quad\text{and}\quad k_{\scriptscriptstyle L}^{\scriptscriptstyle S}=\frac{b_{\scriptscriptstyle R}}{b_{\scriptscriptstyle L}}\alpha_6.$$
\end{proof}

\begin{remark}
	Suppose that the system $\eqref{eq:01}|_{\epsilon=0}$ is of type SCS and that the ordinates of the points $P_{\scriptscriptstyle L}^{\scriptscriptstyle u}$ and $P_{\scriptscriptstyle R}^{\scriptscriptstyle s}$ are equals, see Fig. \ref{fig:03} (b). Then the first order Melnikov function \eqref{eq:melscs}, is given by
	\begin{equation}\label{eq:scs1}
		M(h)=k_0f_0(h)+k_{\scriptscriptstyle C}^{\scriptscriptstyle C}f_{\scriptscriptstyle C}^{\scriptscriptstyle C}(h)+k_{\scriptscriptstyle R}^{\scriptscriptstyle S} f_{\scriptscriptstyle R}^{\scriptscriptstyle S}(h),
	\end{equation}
	with
	\begin{equation*}
	\begin{aligned}
	k_0=\,\,& p_{\scriptscriptstyle R}-u_{\scriptscriptstyle R}+2r_{\scriptscriptstyle R}+2b_{\scriptscriptstyle R}\Big(\frac{p_{\scriptscriptstyle L}-r_{\scriptscriptstyle L}}{b_{\scriptscriptstyle L}}+u_{\scriptscriptstyle C}-p_{\scriptscriptstyle C}\Big)\\
	&+(p_{\scriptscriptstyle R}+u_{\scriptscriptstyle R})\Bigg(\frac{a_{\scriptscriptstyle R}^2-b_{\scriptscriptstyle R}\beta_{\scriptscriptstyle R}}{\omega_{\scriptscriptstyle RS}^2}\Bigg)
	+(p_{\scriptscriptstyle L}+u_{\scriptscriptstyle L})\Bigg(\frac{a_{\scriptscriptstyle R}^2-b_{\scriptscriptstyle R}\beta_{\scriptscriptstyle R}-\omega_{\scriptscriptstyle RS}^2}{\omega_{\scriptscriptstyle LS}\omega_{\scriptscriptstyle RS}}\Bigg),
	\end{aligned}
	\end{equation*}
		
	$$k_{\scriptscriptstyle C}^{\scriptscriptstyle C}=b_{\scriptscriptstyle R}(p_{\scriptscriptstyle C}+u_{\scriptscriptstyle C})\quad\text{and}\quad k_{\scriptscriptstyle R}^{\scriptscriptstyle S}=\frac{\omega_{\scriptscriptstyle LS}(p_{\scriptscriptstyle R}+u_{\scriptscriptstyle R})+\omega_{\scriptscriptstyle RS}(p_{\scriptscriptstyle L}+{u_{\scriptscriptstyle L}})}{2b_{\scriptscriptstyle R}\omega_{\scriptscriptstyle LS}\omega_{\scriptscriptstyle RS}^3},$$
	where the functions $f_0,f_{\scriptscriptstyle C}^{\scriptscriptstyle C},f_{\scriptscriptstyle R}^{\scriptscriptstyle S}$ are the ones defined in \eqref{eq:func}. In fact, if the ordinates of the points $P_{\scriptscriptstyle L}^{\scriptscriptstyle u}$ and $P_{\scriptscriptstyle R}^{\scriptscriptstyle s}$ are equals, then 
	$$\frac{a_{\scriptscriptstyle L}^2+b_{\scriptscriptstyle L} \beta_{\scriptscriptstyle L}-\omega_{\scriptscriptstyle LS}^2}{b_{\scriptscriptstyle L}\omega_{\scriptscriptstyle LS}}=\frac{a_{\scriptscriptstyle R}^2-b_{\scriptscriptstyle R} \beta_{\scriptscriptstyle R}-\omega_{\scriptscriptstyle RS}^2}{b_{\scriptscriptstyle R}\omega_{\scriptscriptstyle RS}}.$$
	Isolating the parameter $\beta_{\scriptscriptstyle L}$ in the equality above, i.e.
	$$\beta_{\scriptscriptstyle L}=\frac{a_{\scriptscriptstyle R}^2 b_{\scriptscriptstyle L}\omega_{\scriptscriptstyle LS}-b_{\scriptscriptstyle L}b_{\scriptscriptstyle R}\beta_{\scriptscriptstyle R}\omega_{\scriptscriptstyle LS}-a_{\scriptscriptstyle L}^2b_{\scriptscriptstyle R}\omega_{\scriptscriptstyle RS}+b_{\scriptscriptstyle R}\omega_{\scriptscriptstyle LS}^2\omega_{\scriptscriptstyle RS}-b_{\scriptscriptstyle L}\omega_{\scriptscriptstyle LS}\omega_{\scriptscriptstyle RS}^2}{b_{\scriptscriptstyle L}b_{\scriptscriptstyle R}\omega_{\scriptscriptstyle RS}},$$
	and replacing on function $M(h)$ given by \eqref{eq:melscs} we obtain the expression \eqref{eq:scs1}.
\end{remark}		

\medskip

The next two theorems provide expressions for the Melnikov function in the cases CCS and CCC. The proof of these results is analogous to proof of Theorem \ref{theo:scs}.

\begin{theorem}\label{theo:ccs}
	Suppose that systems $\eqref{eq:01}|_{\epsilon=0}$ is of the type CCS. Then the first order  Melnikov function $M(h)$ associated to system \eqref{eq:01} can be expressed as	
	\begin{equation}\label{eq:melccs}
		M(h)=k_0f_0(h)+k_{\scriptscriptstyle C}^{\scriptscriptstyle C}f_{\scriptscriptstyle C}^{\scriptscriptstyle C}(h)+k_{\scriptscriptstyle R}^{\scriptscriptstyle C}f_{\scriptscriptstyle R}^{\scriptscriptstyle C}(h)+k_{\scriptscriptstyle L}^{\scriptscriptstyle S}f_{\scriptscriptstyle L}^{\scriptscriptstyle S}(h),
	\end{equation}
	for $h\in(0,\tau)$, where the functions $f_0,f_{\scriptscriptstyle C}^{\scriptscriptstyle C},f_{\scriptscriptstyle R}^{\scriptscriptstyle C},f_{\scriptscriptstyle L}^{\scriptscriptstyle S}$ are the ones defined in \eqref{eq:func}. Here the coefficients $k_0$ and $k_i^j$, for $i=L,C,R$ and $j=C,S$, depend on the parameters of system \eqref{eq:01}.
\end{theorem}

\medskip

\begin{theorem}\label{theo:ccc} 
	Suppose the systems $\eqref{eq:01}|_{\epsilon=0}$ is of the type CCC. Then the first order Melnikov function $M(h)$ associated to system \eqref{eq:01} can be expressed as 	
	\begin{equation}\label{eq:melccc}
		M(h)=k_0f_0(h)+k_{\scriptscriptstyle C}^{\scriptscriptstyle C}f_{\scriptscriptstyle C}^{\scriptscriptstyle C}(h)+k_{\scriptscriptstyle R}^{\scriptscriptstyle C}f_{\scriptscriptstyle R}^{\scriptscriptstyle C}(h)+k_{\scriptscriptstyle L}^{\scriptscriptstyle C}f_{\scriptscriptstyle L}^{\scriptscriptstyle C}(h),
	\end{equation}
	for $h\in(0,\infty)$, where the functions $f_0,f_{\scriptscriptstyle C}^{\scriptscriptstyle C},f_{\scriptscriptstyle R}^{\scriptscriptstyle C},f_{\scriptscriptstyle L}^{\scriptscriptstyle C}$ are the ones defined in \eqref{eq:func}. Here the coefficients $k_0$ and $k_i^{\scriptscriptstyle C}$, for $i=L,C,R$, depend on the parameters of system \eqref{eq:01}.
\end{theorem}

\bigskip

The next corollaries provided lower bounds for the number of limit cycles of system $\eqref{eq:01}$ that can bifurcate from a periodic annulus of system $\eqref{eq:01}|_{\epsilon=0}$ in the cases SCS, CCS and CCC. But first, we will recall basic linear algebra results.


Let $\{f_0, f_1,\dots, f_n\}$ a set of real functions defined on a proper interval $I\subset \mathbb{R}$. We say that $\{f_0, f_1,\dots, f_n\}$ is {\it linearly independent} if the unique solution of the equation
$$\sum_{i=0}^{n}\alpha_i f_i(t)=0,$$
for all $t\in I$, is $\alpha_0=\alpha_1=\dots=\alpha_n=0$.

\begin{proposition}\label{prop:li1}
	If $\{f_0, f_1,\dots, f_n\}$ is linearly independent then there exist $t_1, t_2,\dots, t_n\in I$, with $t_i\ne t_j$ for $i\ne j$, and $\alpha_0, \alpha_1, \dots,\alpha_n\in\mathbb{R}$, not all null, such that for every $j\in \{1, 2, \dots, n\}$
	$$\sum_{i=0}^{n}\alpha_i f_i(t_j)=0.$$
\end{proposition}

For a proof of Proposition \ref{prop:li1}, see for instance \cite{Lli11}.

\medskip

Recall that if the Wronskian,  
\begin{displaymath}
	W(f_0,f_1,\ldots,f_n)(x)=\left|\begin{array}{cccc}
		f_0(x) & f_1(x) & \ldots & f_n(x)\\
		f'_0(x) & f'_1(x) & \ldots & f'_n(x)\\
		\vdots & \vdots & \ddots & \vdots\\
		f^{(n)}_0(x) & f^{(n)}_1(x) & \ldots & f^{(n)}_n(x)
	\end{array}\right|,
\end{displaymath}
where $\{f_0, f_1,\dots, f_n\}$ is a set of functions with derivatives until order $n$ on $I$, is different of zero for some $x\in I$, then $\{f_0, f_1, \dots , f_n\}$ is linearly independent on $I$.

\medskip

\begin{corollary}[Case SCS -- Fig. \ref{fig:03} (a)]\label{scs-a}
	Consider the system \eqref{eq:01} with $a_{\scriptscriptstyle L}=b_{\scriptscriptstyle L}=b_{\scriptscriptstyle R}=c_{\scriptscriptstyle R}=1$, $a_{\scriptscriptstyle R}=c_{\scriptscriptstyle L}=0$, $\beta_{\scriptscriptstyle L}=2$ and  $\beta_{\scriptscriptstyle R}=-2$. Then, for $0<\epsilon<<1$, the system \eqref{eq:01} has at least three limit cycles.	
\end{corollary}	
\begin{proof}
	For $a_{\scriptscriptstyle L}=b_{\scriptscriptstyle L}=b_{\scriptscriptstyle R}=c_{\scriptscriptstyle R}=1$, $a_{\scriptscriptstyle R}=c_{\scriptscriptstyle L}=0$, $\beta_{\scriptscriptstyle L}=2$ and  $\beta_{\scriptscriptstyle R}=-2$, the eigenvalues of the linear part of the left, central and right subsystem from $\eqref{eq:01}|_{\epsilon=0}$ are $\pm 1$, $\pm i$ and $\pm 1$, respectively, i.e. we have one center and two saddles. Moreover, the coordinates of the equilibrium points on the left and right subsystem from $\eqref{eq:01}|_{\epsilon=0}$ are $(-3,2)$ and $(2,0)$, respectively, and so the saddles are real. Note that $P_{\scriptscriptstyle L}^{\scriptscriptstyle u}=(-1,2)$ and $P_{\scriptscriptstyle R}^{\scriptscriptstyle s}=(1,1)$, i.e. the ordinates of this points are different and $\tau=1$. Therefore, the first order Melnikov function from Theorem \ref{theo:scs} becomes
	$$M(h)=k_0f_0(h)+k_{\scriptscriptstyle C}^{\scriptscriptstyle C}f_{\scriptscriptstyle C}^{\scriptscriptstyle C}(h)+k_{\scriptscriptstyle R}^{\scriptscriptstyle S}f_{\scriptscriptstyle R}^{\scriptscriptstyle S}(h)+k_{\scriptscriptstyle L}^{\scriptscriptstyle S}f_{\scriptscriptstyle L}^{\scriptscriptstyle S}(h),\quad \forall\, h\in(0,1),$$
	with
	\begin{equation*}
		\begin{aligned}
			f_0(h)&= h, \\
			f_{\scriptscriptstyle C}^{\scriptscriptstyle C}(h)&=(h^2+1)\arccos\bigg(\frac{h^2-1}{h^2+1}\bigg),\\
			f_{\scriptscriptstyle R}^{\scriptscriptstyle S}(h)&=(h^2-1)\log\bigg(-\frac{h+1}{h-1}\bigg),\\
			f_{\scriptscriptstyle L}^{\scriptscriptstyle S}(h)&=(h^2-4)\log\bigg(-\frac{h+2}{h-2}\bigg),	
		\end{aligned}
	\end{equation*}	
	$$k_0=2(2p_{\scriptscriptstyle L}-p_{\scriptscriptstyle C}-r_{\scriptscriptstyle L}+r_{\scriptscriptstyle R}+u_{\scriptscriptstyle C}+u_{\scriptscriptstyle L})+3p_{\scriptscriptstyle R}+u_{\scriptscriptstyle R},\quad k_{\scriptscriptstyle C}^{\scriptscriptstyle C}=p_{\scriptscriptstyle C}+u_{\scriptscriptstyle C},$$
	$$ k_{\scriptscriptstyle R}^{\scriptscriptstyle S}=\frac{p_{\scriptscriptstyle R}+u_{\scriptscriptstyle R}}{2}\quad\text{and}\quad k_{\scriptscriptstyle L}^{\scriptscriptstyle S}=\frac{p_{\scriptscriptstyle L}+u_{\scriptscriptstyle L}}{2}.$$	
	Consider the set of functions $\mathcal{F}_{\scriptscriptstyle SCS}=\{f_0,f_{\scriptscriptstyle C}^{\scriptscriptstyle C},f_{\scriptscriptstyle R}^{\scriptscriptstyle S},f_{\scriptscriptstyle L}^{\scriptscriptstyle S}\}$. Using the algebraic manipulator Mathematica (see \cite{Wol20}), we compute the Wronskian $W(f_0, f_{\scriptscriptstyle C}^{\scriptscriptstyle C},f_{\scriptscriptstyle R}^{\scriptscriptstyle S},f_{\scriptscriptstyle L}^{\scriptscriptstyle S})(h)$ and evaluate in some point on the interval $(0,1)$. More precisely, $W(f_0, f_{\scriptscriptstyle C}^{\scriptscriptstyle C},f_{\scriptscriptstyle R}^{\scriptscriptstyle S},f_{\scriptscriptstyle L}^{\scriptscriptstyle S})(0.4)=9.16568$. Then, set of functions $\mathcal{F}_{\scriptscriptstyle SCS}$ is linearly independent on the interval  $(0,1)$ and, by Proposition \ref{prop:li1}, there are $h_i\in(0,1)$, with $i=1,2,3$, such that $M(h_i)=0$. Therefore, for $0< \epsilon <<1$, the system \eqref{eq:01} has a unique limit cycle near $L_{h_i}$, for each $i=1,2,3$, i.e. the system \eqref{eq:01} has at least three limit cycles.
\end{proof}	

\begin{corollary}[Case SCS -- Fig. \ref{fig:03} (b)]\label{scs-b}
	Consider the system \eqref{eq:01} with $a_{\scriptscriptstyle L}=b_{\scriptscriptstyle L}=b_{\scriptscriptstyle R}=c_{\scriptscriptstyle R}=\beta_{\scriptscriptstyle L}=1$, $a_{\scriptscriptstyle R}=c_{\scriptscriptstyle L}=0$ and $\beta_{\scriptscriptstyle R}=-2$. Then, for $0<\epsilon<<1$, the system \eqref{eq:01} has at least two limit cycles.	
\end{corollary}	
\begin{proof}
	For $a_{\scriptscriptstyle L}=b_{\scriptscriptstyle L}=b_{\scriptscriptstyle R}=c_{\scriptscriptstyle R}=\beta_{\scriptscriptstyle L}=1$, $a_{\scriptscriptstyle R}=c_{\scriptscriptstyle L}=0$ and $\beta_{\scriptscriptstyle R}=-2$, the eigenvalues of the linear part of the left, central and right subsystem from $\eqref{eq:01}|_{\epsilon=0}$ are $\pm 1$, $\pm i$ and $\pm 1$, respectively, i.e. we have one center and two saddles. Moreover, the coordinates of the equilibrium points on the left and right subsystem from $\eqref{eq:01}|_{\epsilon=0}$ are $(-2,1)$ and $(2,0)$, respectively, and so the saddles are real. Note that $P_{\scriptscriptstyle L}^{\scriptscriptstyle u}=(-1,1)$ and $P_{\scriptscriptstyle R}^{\scriptscriptstyle s}=(1,1)$, i.e. the ordinates of this points are the same and $\tau=1$. Therefore, the first order Melnikov function from Theorem \ref{theo:scs} becomes
	$$M(h)=k_0f_0(h)+k_{\scriptscriptstyle C}^{\scriptscriptstyle C}f_{\scriptscriptstyle C}^{\scriptscriptstyle C}(h)+k_{\scriptscriptstyle R}^{\scriptscriptstyle S}f_{\scriptscriptstyle R}^{\scriptscriptstyle S}(h),\quad \forall \,h\in(0,1),$$
	with
	\begin{equation*}
		\begin{aligned}
			f_0(h)&= h, \\
			f_{\scriptscriptstyle C}^{\scriptscriptstyle C}(h)&=(h^2+1)\arccos\bigg(\frac{h^2-1}{h^2+1}\bigg),\\
			f_{\scriptscriptstyle R}^{\scriptscriptstyle S}(h)&=(h^2-1)\log\bigg(-\frac{h+1}{h-1}\bigg),	
		\end{aligned}
	\end{equation*}	
	$$k_0=2(u_{\scriptscriptstyle C}-p_{\scriptscriptstyle C}-r_{\scriptscriptstyle L}+r_{\scriptscriptstyle R})+3(p_{\scriptscriptstyle R}+p_{\scriptscriptstyle L})+u_{\scriptscriptstyle R}+u_{\scriptscriptstyle L},\quad k_{\scriptscriptstyle C}^{\scriptscriptstyle C}=p_{\scriptscriptstyle C}+u_{\scriptscriptstyle C}$$
	and
	$$ k_{\scriptscriptstyle R}^{\scriptscriptstyle S}=\frac{p_{\scriptscriptstyle R}+p_{\scriptscriptstyle L}+u_{\scriptscriptstyle R}+u_{\scriptscriptstyle L}}{2}.$$	
	Consider the set of functions $\mathcal{F}_{\scriptscriptstyle SCS}=\{f_0,f_{\scriptscriptstyle C}^{\scriptscriptstyle C},f_{\scriptscriptstyle R}^{\scriptscriptstyle S}\}$. Using the algebraic manipulator Mathematica, we compute the Wronskian $W(f_0, f_{\scriptscriptstyle C}^{\scriptscriptstyle C},f_{\scriptscriptstyle R}^{\scriptscriptstyle S})(h)$ and evaluate in some point on the interval $(0,1)$. More precisely, $W(f_0, f_{\scriptscriptstyle C}^{\scriptscriptstyle C},f_{\scriptscriptstyle R}^{\scriptscriptstyle S})(0.4)=-10.6955$. Then, set of functions $\mathcal{F}_{\scriptscriptstyle SCS}$ is linearly independent on the interval  $(0,1)$ and, by Proposition \ref{prop:li1}, there are $h_i\in(0,1)$, with $i=1,2$, such that $M(h_i)=0$. Therefore, for $0< \epsilon <<1$, the system \eqref{eq:01} has a unique limit cycle near $L_{h_i}$, for each $i=1,2$, i.e. the system \eqref{eq:01} has at least two limit cycles.
\end{proof}	

\begin{corollary}[Case CCS -- Fig. \ref{fig:03} (c)] \label{ccs-c}
	Consider the system \eqref{eq:01} with $a_{\scriptscriptstyle L}=\beta_{\scriptscriptstyle L}=b_{\scriptscriptstyle R}=c_{\scriptscriptstyle R}=1$, $b_{\scriptscriptstyle L}=2$, $c_{\scriptscriptstyle L}=-1$, $a_{\scriptscriptstyle R}=0$ and $\beta_{\scriptscriptstyle R}=-2$. Then for $0<\epsilon<<1$, the system \eqref{eq:01} has at least three limit cycles.	
\end{corollary}	
\begin{proof}
	For $a_{\scriptscriptstyle L}=\beta_{\scriptscriptstyle L}=b_{\scriptscriptstyle R}=c_{\scriptscriptstyle R}=1$, $b_{\scriptscriptstyle L}=2$, $c_{\scriptscriptstyle L}=-1$, $a_{\scriptscriptstyle R}=0$ and $\beta_{\scriptscriptstyle R}=-2$, the eigenvalues of the linear part of left, central and right subsystem from $\eqref{eq:01}|_{\epsilon=0}$ are $\pm i$, $\pm i$ and $\pm 1$, respectively, i.e. we have two centers and one saddle. Moreover, the coordinates of the equilibrium point on the left and right subsystem from $\eqref{eq:01}|_{\epsilon=0}$ are  $(3,-2)$ and $(2,0)$, respectively, and so we have a virtual center and a real saddle. Note that $P_{\scriptscriptstyle R}^{\scriptscriptstyle s}=(1,1)$ and $\tau=1$. Therefore, the first order Melnikov function from Theorem \ref{theo:ccs} becomes
	$$M(h)=k_0f_0(h)+k_{\scriptscriptstyle C}^{\scriptscriptstyle C}f_{\scriptscriptstyle C}^{\scriptscriptstyle C}(h)+k_{\scriptscriptstyle R}^{\scriptscriptstyle S}f_{\scriptscriptstyle R}^{\scriptscriptstyle S}(h)+k_{\scriptscriptstyle L}^{\scriptscriptstyle C}f_{\scriptscriptstyle L}^{\scriptscriptstyle C}(h),\quad \forall \,h\in(0,1),$$
	with
	\begin{equation*}
		\begin{aligned}
			f_0(h)&= h, \\
			f_{\scriptscriptstyle C}^{\scriptscriptstyle C}(h)&=(h^2+1)\arccos\bigg(\frac{h^2-1}{h^2+1}\bigg),\\
			f_{\scriptscriptstyle R}^{\scriptscriptstyle S}(h)&=(h^2-1)\log\bigg(-\frac{h+1}{h-1}\bigg),\\
			f_{\scriptscriptstyle L}^{\scriptscriptstyle C}(h)&=(h^2+4)\arccos\bigg(\frac{4-h^2}{h^2+4}\bigg),	
		\end{aligned}
	\end{equation*}
	$$k_0=2(u_{\scriptscriptstyle C}-p_{\scriptscriptstyle C}+r_{\scriptscriptstyle R}-u_{\scriptscriptstyle L})-r_{\scriptscriptstyle L}-p_{\scriptscriptstyle L}+u_{\scriptscriptstyle R}+3p_{\scriptscriptstyle R},\quad k_{\scriptscriptstyle C}^{\scriptscriptstyle C}=p_{\scriptscriptstyle C}+u_{\scriptscriptstyle C},$$
	$$ k_{\scriptscriptstyle R}^{\scriptscriptstyle S}=\frac{p_{\scriptscriptstyle R}+u_{\scriptscriptstyle R}}{2}\quad\text{and}\quad k_{\scriptscriptstyle L}^{\scriptscriptstyle S}=\frac{p_{\scriptscriptstyle L}+u_{\scriptscriptstyle L}}{2}.$$		
	Consider the set of functions $\mathcal{F}_{\scriptscriptstyle CCS}=\{f_0,f_{\scriptscriptstyle C}^{\scriptscriptstyle C},f_{\scriptscriptstyle R}^{\scriptscriptstyle S},f_{\scriptscriptstyle L}^{\scriptscriptstyle C}\}$. Using the algebraic manipulator Mathematica, we compute the Wronskian 
	$W(f_0, f_{\scriptscriptstyle C}^{\scriptscriptstyle C},$ $f_{\scriptscriptstyle R}^{\scriptscriptstyle S},f_{\scriptscriptstyle L}^{\scriptscriptstyle C})(h)$ and evaluate in some point on the interval $(0,1)$. More precisely, $W(f_0, f_{\scriptscriptstyle C}^{\scriptscriptstyle C},f_{\scriptscriptstyle R}^{\scriptscriptstyle S},f_{\scriptscriptstyle L}^{\scriptscriptstyle C})(0.4)=13.25$. Then, set of functions $\mathcal{F}_{\scriptscriptstyle CCS}$ is linearly independent on the interval $(0,1)$ and, by Proposition \ref{prop:li1}, there are $h_i\in(0,1)$, with $i=1,2,3$, such that $M(h_i)=0$. Therefore, for $0< \epsilon <<1$, the system \eqref{eq:01} has a unique limit cycle near $L_{h_i}$, for each $i=1,2,3$, i.e. the system \eqref{eq:01} has at least three limit cycles.
\end{proof}	

\begin{corollary}[Case CCS -- Fig. \ref{fig:03} (d)]\label{ccs-d}
	Consider the system \eqref{eq:01} with $a_{\scriptscriptstyle L}=b_{\scriptscriptstyle R}=c_{\scriptscriptstyle R}=1$, $b_{\scriptscriptstyle L}=2$, $c_{\scriptscriptstyle L}=-1$, $a_{\scriptscriptstyle R}=0$ and $\beta_{\scriptscriptstyle R}=\beta_{\scriptscriptstyle L}=-2$. Then for $0<\epsilon<<1$, the system \eqref{eq:01} has at least three limit cycles.	
\end{corollary}	
\begin{proof}
	For $a_{\scriptscriptstyle L}=b_{\scriptscriptstyle R}=c_{\scriptscriptstyle R}=1$, $b_{\scriptscriptstyle L}=2$, $c_{\scriptscriptstyle L}=-1$, $a_{\scriptscriptstyle R}=0$ and $\beta_{\scriptscriptstyle R}=\beta_{\scriptscriptstyle L}=-2$, the eigenvalues of the linear part of left, central and right subsystem from $\eqref{eq:01}|_{\epsilon=0}$ are $\pm i$, $\pm i$ and $\pm 1$, respectively, i.e. we have two centers and one saddle. Moreover, the coordinates of the equilibrium point on the left and right subsystem from $\eqref{eq:01}|_{\epsilon=0}$ are  $(-3,1)$ and $(2,0)$, respectively, and so we have a real center and a real saddle. Note that $P_{\scriptscriptstyle R}^{\scriptscriptstyle s}=(1,1)$ and $\tau=1$. Therefore, the first order Melnikov function from Theorem \ref{theo:ccs} becomes
	$$M(h)=k_0f_0(h)+k_{\scriptscriptstyle C}^{\scriptscriptstyle C}f_{\scriptscriptstyle C}^{\scriptscriptstyle C}(h)+k_{\scriptscriptstyle R}^{\scriptscriptstyle S}f_{\scriptscriptstyle R}^{\scriptscriptstyle S}(h)+k_{\scriptscriptstyle L}^{\scriptscriptstyle C}f_{\scriptscriptstyle L}^{\scriptscriptstyle C}(h),\quad \forall \,h\in(0,1),$$
	with
	\begin{equation*}
		\begin{aligned}
			f_0(h)&= h, \\
			f_{\scriptscriptstyle C}^{\scriptscriptstyle C}(h)&=(h^2+1)\arccos\bigg(\frac{h^2-1}{h^2+1}\bigg),\\
			f_{\scriptscriptstyle R}^{\scriptscriptstyle S}(h)&=(h^2-1)\log\bigg(-\frac{h+1}{h-1}\bigg),\\
			f_{\scriptscriptstyle L}^{\scriptscriptstyle C}(h)&=(h^2+4)\arccos\bigg(\frac{4-h^2}{h^2+4}\bigg),	
		\end{aligned}
	\end{equation*}
	$$k_0=2(u_{\scriptscriptstyle C}-p_{\scriptscriptstyle C}+r_{\scriptscriptstyle R}+p_{\scriptscriptstyle L})-r_{\scriptscriptstyle L}+u_{\scriptscriptstyle L}+u_{\scriptscriptstyle R}+3p_{\scriptscriptstyle R},\quad k_{\scriptscriptstyle C}^{\scriptscriptstyle C}=p_{\scriptscriptstyle C}+u_{\scriptscriptstyle C},$$
	$$k_{\scriptscriptstyle R}^{\scriptscriptstyle S}=\frac{p_{\scriptscriptstyle R}+u_{\scriptscriptstyle R}}{2}\quad\text{and}\quad k_{\scriptscriptstyle L}^{\scriptscriptstyle S}=\frac{p_{\scriptscriptstyle L}+u_{\scriptscriptstyle L}}{2}.$$		
	Consider the set of functions $\mathcal{F}_{\scriptscriptstyle CCS}=\{f_0,f_{\scriptscriptstyle C}^{\scriptscriptstyle C},f_{\scriptscriptstyle R}^{\scriptscriptstyle S},f_{\scriptscriptstyle L}^{\scriptscriptstyle C}\}$. Using the algebraic manipulator Mathematica, we compute the Wronskian 
	$W(f_0, f_{\scriptscriptstyle C}^{\scriptscriptstyle C},$ $f_{\scriptscriptstyle R}^{\scriptscriptstyle S},f_{\scriptscriptstyle L}^{\scriptscriptstyle C})(h)$ and evaluate in some point on the interval $(0,1)$. More precisely, $W(f_0, f_{\scriptscriptstyle C}^{\scriptscriptstyle C},f_{\scriptscriptstyle R}^{\scriptscriptstyle S},f_{\scriptscriptstyle L}^{\scriptscriptstyle C})(0.2)=-4.26846$. Then, set of functions $\mathcal{F}_{\scriptscriptstyle CCS}$ is linearly independent on the interval $(0,1)$ and, by Proposition \ref{prop:li1}, there are $h_i\in(0,1)$, with $i=1,2,3$, such that $M(h_i)=0$. Therefore, for $0< \epsilon <<1$, the system \eqref{eq:01} has a unique limit cycle near $L_{h_i}$, for each $i=1,2,3$, i.e. the system \eqref{eq:01} has at least three limit cycles.
\end{proof}

\begin{corollary}[Case CCC -- Fig. \ref{fig:04} (a)]\label{ccc-a}
	Consider the system \eqref{eq:01} with $a_{\scriptscriptstyle L}=b_{\scriptscriptstyle R}=\beta_{\scriptscriptstyle L}=1$, $b_{\scriptscriptstyle L}=2$, $c_{\scriptscriptstyle L}=c_{\scriptscriptstyle R}=-1$ and $a_{\scriptscriptstyle R}=\beta_{\scriptscriptstyle R}=0$. Then, for $0<\epsilon<<1$, the system \eqref{eq:01} has at least three limit cycles.	
\end{corollary}	
\begin{proof}
	For $a_{\scriptscriptstyle L}=b_{\scriptscriptstyle R}=\beta_{\scriptscriptstyle L}=1$, $b_{\scriptscriptstyle L}=2$, $c_{\scriptscriptstyle L}=c_{\scriptscriptstyle R}=-1$ and $a_{\scriptscriptstyle R}=\beta_{\scriptscriptstyle R}=0$, the eigenvalues of the linear part of the left, central and right subsystem from $\eqref{eq:01}|_{\epsilon=0}$ are $\pm i$, $\pm i$ and $\pm i$, respectively, i.e. we have three centers. Moreover, the coordinates of the equilibrium points of the left and right subsystem from $\eqref{eq:01}|_{\epsilon=0}$ are $(3,-2)$ and $(0,0)$, respectively, and so the centers are virtual. Therefore, the first order Melnikov function from Theorem \ref{theo:ccc} becomes
	$$M(h)=k_0f_0(h)+k_{\scriptscriptstyle C}^{\scriptscriptstyle C}f_{\scriptscriptstyle C}^{\scriptscriptstyle C}(h)+k_{\scriptscriptstyle R}^{\scriptscriptstyle C}f_{\scriptscriptstyle R}^{\scriptscriptstyle C}(h)+k_{\scriptscriptstyle L}^{\scriptscriptstyle C}f_{\scriptscriptstyle L}^{\scriptscriptstyle C}(h),\quad \forall \,h\in(0,\infty),$$
	with
	\begin{equation*}
		\begin{aligned}
			f_0(h)&= h, \\
			f_{\scriptscriptstyle C}^{\scriptscriptstyle C}(h)&=(h^2+1)\arccos\bigg(\frac{h^2-1}{h^2+1}\bigg),\\
			f_{\scriptscriptstyle R}^{\scriptscriptstyle C}(h)&=(h^2-1)\arccos\bigg(\frac{1-h^2}{h^2+1}\bigg),\\
			f_{\scriptscriptstyle L}^{\scriptscriptstyle C}(h)&=(h^2+4)\arccos\bigg(\frac{4-h^2}{h^2+4}\bigg),	
		\end{aligned}
	\end{equation*}	
	$$k_0=2(u_{\scriptscriptstyle C}-p_{\scriptscriptstyle C}+r_{\scriptscriptstyle R}-u_{\scriptscriptstyle L})-r_{\scriptscriptstyle L}-p_{\scriptscriptstyle L}-u_{\scriptscriptstyle R}+p_{\scriptscriptstyle R},\quad k_{\scriptscriptstyle C}^{\scriptscriptstyle C}=p_{\scriptscriptstyle C}+u_{\scriptscriptstyle C},$$
	$$ k_{\scriptscriptstyle R}^{\scriptscriptstyle S}=\frac{p_{\scriptscriptstyle R}+u_{\scriptscriptstyle R}}{2}\quad\text{and}\quad k_{\scriptscriptstyle L}^{\scriptscriptstyle S}=\frac{p_{\scriptscriptstyle L}+u_{\scriptscriptstyle L}}{2}.$$	
	Consider the set of functions $\mathcal{F}_{\scriptscriptstyle CCC}=\{f_0,f_{\scriptscriptstyle C}^{\scriptscriptstyle C},f_{\scriptscriptstyle R}^{\scriptscriptstyle C},f_{\scriptscriptstyle L}^{\scriptscriptstyle C}\}$. Using the algebraic manipulator Mathematica, we compute the Wronskian
	$W(f_0, f_{\scriptscriptstyle C}^{\scriptscriptstyle C},$ $f_{\scriptscriptstyle R}^{\scriptscriptstyle C},f_{\scriptscriptstyle L}^{\scriptscriptstyle C})(h)$ and evaluate in some point on the interval $(0,\infty)$. More precisely, $W(f_0, f_{\scriptscriptstyle C}^{\scriptscriptstyle C},f_{\scriptscriptstyle R}^{\scriptscriptstyle C},f_{\scriptscriptstyle L}^{\scriptscriptstyle C})(0.2)=-2.92151$. Then, set of functions $\mathcal{F}_{\scriptscriptstyle CCC}$ is linearly independent on the interval $(0,\infty)$ and, by Proposition \ref{prop:li1}, there are $h_i\in(0,\infty)$, with $i=1,2,3$, such that $M(h_i)=0$. Therefore, for $0< \epsilon <<1$, the system \eqref{eq:01} has a unique limit cycle near $L_{h_i}$, for each $i=1,2,3$, i.e. the system \eqref{eq:01} has at least three limit cycles.
\end{proof}

\begin{corollary}[Case CCC -- Fig. \ref{fig:04} (b)]\label{ccc-b}
	Consider the system \eqref{eq:01} with $a_{\scriptscriptstyle L}=b_{\scriptscriptstyle R}=1$, $\beta_{\scriptscriptstyle L}=-3$, $b_{\scriptscriptstyle L}=2$, $c_{\scriptscriptstyle L}=c_{\scriptscriptstyle R}=-1$ and $a_{\scriptscriptstyle R}=\beta_{\scriptscriptstyle R}=0$. Then, for $0<\epsilon<<1$, the system \eqref{eq:01} has at least three limit cycles.	
\end{corollary}	
\begin{proof}
	For $a_{\scriptscriptstyle L}=b_{\scriptscriptstyle R}=1$, $\beta_{\scriptscriptstyle L}=-3$, $b_{\scriptscriptstyle L}=2$, $c_{\scriptscriptstyle L}=c_{\scriptscriptstyle R}=-1$ and $a_{\scriptscriptstyle R}=\beta_{\scriptscriptstyle R}=0$, the eigenvalues of the linear part of the left, central and right subsystem from $\eqref{eq:01}|_{\epsilon=0}$ are $\pm i$, $\pm i$ and $\pm i$, respectively, i.e. we have three centers. Moreover, the coordinates of the equilibrium points of the left and right subsystem from $\eqref{eq:01}|_{\epsilon=0}$ are $(-5,2)$ and $(0,0)$, respectively, and so we have a real center and a virtual center. Therefore, the first order Melnikov function from Theorem \ref{theo:ccc} becomes
	$$M(h)=k_0f_0(h)+k_{\scriptscriptstyle C}^{\scriptscriptstyle C}f_{\scriptscriptstyle C}^{\scriptscriptstyle C}(h)+k_{\scriptscriptstyle R}^{\scriptscriptstyle C}f_{\scriptscriptstyle R}^{\scriptscriptstyle C}(h)+k_{\scriptscriptstyle L}^{\scriptscriptstyle C}f_{\scriptscriptstyle L}^{\scriptscriptstyle C}(h),\quad \forall \,h\in(0,\infty),$$
	with
	\begin{equation*}
		\begin{aligned}
			f_0(h)&= h, \\
			f_{\scriptscriptstyle C}^{\scriptscriptstyle C}(h)&=(h^2+1)\arccos\bigg(\frac{h^2-1}{h^2+1}\bigg),\\
			f_{\scriptscriptstyle R}^{\scriptscriptstyle C}(h)&=(h^2-1)\arccos\bigg(\frac{1-h^2}{h^2+1}\bigg),\\
			f_{\scriptscriptstyle L}^{\scriptscriptstyle C}(h)&=(h^2+4)\arccos\bigg(\frac{4-h^2}{h^2+4}\bigg),	
		\end{aligned}
	\end{equation*}	
	$$k_0=2(u_{\scriptscriptstyle C}-p_{\scriptscriptstyle C}+r_{\scriptscriptstyle R}+u_{\scriptscriptstyle L})-r_{\scriptscriptstyle L}-u_{\scriptscriptstyle R}+p_{\scriptscriptstyle R}+3p_{\scriptscriptstyle L},\quad k_{\scriptscriptstyle C}^{\scriptscriptstyle C}=p_{\scriptscriptstyle C}+u_{\scriptscriptstyle C},$$
	$$ k_{\scriptscriptstyle R}^{\scriptscriptstyle S}=\frac{p_{\scriptscriptstyle R}+u_{\scriptscriptstyle R}}{2}\quad\text{and}\quad k_{\scriptscriptstyle L}^{\scriptscriptstyle S}=\frac{p_{\scriptscriptstyle L}+u_{\scriptscriptstyle L}}{2}.$$	
	Consider the set of functions $\mathcal{F}_{\scriptscriptstyle CCC}=\{f_0,f_{\scriptscriptstyle C}^{\scriptscriptstyle C},f_{\scriptscriptstyle R}^{\scriptscriptstyle C},f_{\scriptscriptstyle L}^{\scriptscriptstyle C}\}$. Using the algebraic manipulator Mathematica, we compute the Wronskian 
	$W(f_0, f_{\scriptscriptstyle C}^{\scriptscriptstyle C},$ $f_{\scriptscriptstyle R}^{\scriptscriptstyle C},f_{\scriptscriptstyle L}^{\scriptscriptstyle C})(h)$ and evaluate in some point on the interval $(0,\infty)$. More precisely, $W(f_0, f_{\scriptscriptstyle C}^{\scriptscriptstyle C},f_{\scriptscriptstyle R}^{\scriptscriptstyle C},f_{\scriptscriptstyle L}^{\scriptscriptstyle C})(0.5)=7.2124$. Then, set of functions $\mathcal{F}_{\scriptscriptstyle CCC}$ is linearly independent on the interval $(0,\infty)$ and, by Proposition \ref{prop:li1},  there are $h_i\in(0,\infty)$, with $i=1,2,3$, such that $M(h_i)=0$. Therefore, for $0< \epsilon <<1$, the system \eqref{eq:01} has a unique limit cycle near $L_{h_i}$, for each $i=1,2,3$, i.e. the system \eqref{eq:01} has at least three limit cycles.
\end{proof}	

\begin{corollary}[Case CCC -- Fig. \ref{fig:04} (c)]\label{ccc-c}
	Consider the system \eqref{eq:01} with $a_{\scriptscriptstyle L}=b_{\scriptscriptstyle R}=1$, $\beta_{\scriptscriptstyle L}=-3$, $b_{\scriptscriptstyle L}=\beta_{\scriptscriptstyle R}=2$, $c_{\scriptscriptstyle L}=c_{\scriptscriptstyle R}=-1$ and $a_{\scriptscriptstyle R}=0$. Then, for $0<\epsilon<<1$, the system \eqref{eq:01} has at least three limit cycles.	
\end{corollary}	
\begin{proof}
	For $a_{\scriptscriptstyle L}=b_{\scriptscriptstyle R}=1$, $\beta_{\scriptscriptstyle L}=-3$, $b_{\scriptscriptstyle L}=\beta_{\scriptscriptstyle R}=2$, $c_{\scriptscriptstyle L}=c_{\scriptscriptstyle R}=-1$ and $a_{\scriptscriptstyle R}=0$, the eigenvalues of the linear part of the left, central and right subsystem from $\eqref{eq:01}|_{\epsilon=0}$ are $\pm i$, $\pm i$ and $\pm i$, respectively, i.e. we have three centers. Moreover, the coordinates of the equilibrium points of the left and right subsystem from $\eqref{eq:01}|_{\epsilon=0}$ are $(-5,2)$ and $(2,0)$, respectively, and so the centers are real. Therefore, the first order Melnikov function from Theorem \ref{theo:ccc} becomes
	$$M(h)=k_0f_0(h)+k_{\scriptscriptstyle C}^{\scriptscriptstyle C}f_{\scriptscriptstyle C}^{\scriptscriptstyle C}(h)+k_{\scriptscriptstyle R}^{\scriptscriptstyle C}f_{\scriptscriptstyle R}^{\scriptscriptstyle C}(h)+k_{\scriptscriptstyle L}^{\scriptscriptstyle C}f_{\scriptscriptstyle L}^{\scriptscriptstyle C}(h),\quad \forall \,h\in(0,\infty),$$
	with
	\begin{equation*}
		\begin{aligned}
			f_0(h)&= h, \\
			f_{\scriptscriptstyle C}^{\scriptscriptstyle C}(h)&=(h^2+1)\arccos\bigg(\frac{h^2-1}{h^2+1}\bigg),\\
			f_{\scriptscriptstyle R}^{\scriptscriptstyle C}(h)&=(h^2-1)\arccos\bigg(\frac{1-h^2}{h^2+1}\bigg),\\
			f_{\scriptscriptstyle L}^{\scriptscriptstyle C}(h)&=(h^2+4)\arccos\bigg(\frac{4-h^2}{h^2+4}\bigg),	
		\end{aligned}
	\end{equation*}	
	$$k_0=2(u_{\scriptscriptstyle C}-p_{\scriptscriptstyle C}+r_{\scriptscriptstyle R}+u_{\scriptscriptstyle L})-r_{\scriptscriptstyle L}+u_{\scriptscriptstyle R}+3(p_{\scriptscriptstyle L}+p_{\scriptscriptstyle R}),\quad k_{\scriptscriptstyle C}^{\scriptscriptstyle C}=p_{\scriptscriptstyle C}+u_{\scriptscriptstyle C},$$
	$$ k_{\scriptscriptstyle R}^{\scriptscriptstyle S}=\frac{p_{\scriptscriptstyle R}+u_{\scriptscriptstyle R}}{2}\quad\text{and}\quad k_{\scriptscriptstyle L}^{\scriptscriptstyle S}=\frac{p_{\scriptscriptstyle L}+u_{\scriptscriptstyle L}}{2}.$$	
	Consider the set of functions $\mathcal{F}_{\scriptscriptstyle CCC}=\{f_0,f_{\scriptscriptstyle C}^{\scriptscriptstyle C},f_{\scriptscriptstyle R}^{\scriptscriptstyle C},f_{\scriptscriptstyle L}^{\scriptscriptstyle C}\}$. Using the algebraic manipulator Mathematica, we compute the Wronskian
	$W(f_0, f_{\scriptscriptstyle C}^{\scriptscriptstyle C},$ $f_{\scriptscriptstyle R}^{\scriptscriptstyle C},f_{\scriptscriptstyle L}^{\scriptscriptstyle C})(h)$ and evaluate in some point on the interval $(0,\infty)$. More precisely, $W(f_0, f_{\scriptscriptstyle C}^{\scriptscriptstyle C},f_{\scriptscriptstyle R}^{\scriptscriptstyle C},f_{\scriptscriptstyle L}^{\scriptscriptstyle C})(0.5)=7.2124$. Then, set of functions $\mathcal{F}_{\scriptscriptstyle CCC}$ is linearly independent on the interval $(0,\infty)$ and, by Proposition \ref{prop:li1}, there are $h_i\in(0,\infty)$, with $i=1,2,3$, such that $M(h_i)=0$. Therefore, for $0< \epsilon <<1$, the system \eqref{eq:01} has a unique limit cycle near $L_{h_i}$, for each $i=1,2,3$, i.e. the system \eqref{eq:01} has at least three limit cycles.
\end{proof}

\section{Acknowledgments}

The first author is partially supported by S\~ao Paulo Research Foundation (FAPESP) grants 19/10269-3 and 18/19726-5.
The second author is supported by CAPES grants 88882.434343/2019-01.


\addcontentsline{toc}{chapter}{Bibliografia}
\begin{thebibliography}{80}

\bibitem{And66} Andronov, A., Vitt, A., Khaikin, S.: Theory of Oscillations, Pergamon Press, Oxford (1966).

\bibitem{Buz13} Buzzi, C. A., Pessoa, C., Torregrosa, J.: Piecewise linear perturbations of a linear center, Discrete Contin. Dyn. Syst. {\bf 33}, 3915--3936 (2013).

\bibitem{Bra13} Braga, D. C., Mello, L. F.: Limit cycles in a family of discontinuous piecewise linear differential systems with two zones in the plane, Nonlin. Dyn. {\bf 73}, 1283--1288 (2013).

\bibitem{Chu90} Chua, L. O., Lin, G.: Canonical realization of Chua's circuit family, IEEE Trans. Circuits Syst. {\bf 37}, 885--902 (1990).

\bibitem{diB08} di Bernardo, M., Budd, C. J., Champneys, A. R., Kowalczyk, P.: Piecewise-Smooth Dynamical Systems: Theory and Applications, Springer (2008).

\bibitem{Don17} Dong, G., Liu, C.: Note on limit cycles for m--piecewise discontinuous polynomial Liénard differential equations, Z. Angew. Math. Phys. {\bf 68}, 97 (2017).

\bibitem{Fit61} FitzHugh, R.: Impulses and physiological states
in theoretical models of nerve membrane, Biophys. J. {\bf 1}, 445--466 (1961).

\bibitem{Fon20} Fonseca, A. F., Llibre, J., Mello, L. F.: Limit cycles in planar piecewise linear Hamiltonian systems with three zones without equilibrium points, Int. J. Bifur. Chaos Appl. Sci. Eng. {\bf 30}, 2050157, 8 pp (2020).

\bibitem{Fre98} Freire, E., Ponce, E., Rodrigo, F., Torres, F.: Bifurcation sets of continuous piecewise linear systems with two zones, Int. J. Bifur. Chaos Appl. Sci. Eng. {\bf 8}, 2073--2097 (1998).

\bibitem{Fre12} Freire, E., Ponce E., Torres, F.: Canonical discontinuous planar piecewise linear systems, SIAM J. Appl. Dyn. Syst. {\bf 11}, 181--211 (2012).

\bibitem{Fre14b} Freire, E., Ponce, E., Torres, F.: The discontinuous matching of two planar linear foci can have three nested crossing limit cycles, Publ. Mat. {\bf 2014}, 221--253 (2014).

\bibitem{Hu13} Hu, N., Du, Z.: Bifurcation of periodic orbits emanated from a vertex in discontinuous planar systems, Commun. Nonlinear Sci. Numer. Simulat {\bf 18}, 3436--3448 (2013).

\bibitem{Lim17} Lima, M., Pessoa, C., Pereira, W.: Limit cycles bifurcating from a periodic annulus in continuous piecewise linear differential systems with three zones, Int. J. Bifur. Chaos Appl. Sci. Eng. {\bf 27}, 1750022, 14 pp (2017).

\bibitem{Lli11} Llibre, J., Swirszcz, G.: On the limit cycles of polynomial vector fields,  Dyn. Contin. Discr. Impul. Syst., Ser. A {\bf 18}, 203–214 (2011).

\bibitem{Lli14} Llibre, J., Teruel, E.: Introduction to the Qualitative Theory of Differential Systems; Planar, Symmetric and Continuous Piecewise Linear Systems, Birkhauser (2014).

\bibitem{Lli15} Llibre, J., Ponce, E., Valls, C.: Uniqueness and non--uniqueness of limit cycles for piecewise linear differential systems with three zones and no symmetry, J. Nonlin. Sci. {\bf 25}, 861--887 (2015).

\bibitem{Lli15b} Llibre, J., Teixeira, M. A.: Limit cycles for m--piecewise discontinuous polynomial Liénard differential equations, Z. Angew. Math. Phys. {\bf 66}, 51--66 (2015).

\bibitem{Lli18a} Llibre, J., Teixeira, M. A.: Piecewise linear differential systems with only centers can create limit cycles? Nonlin. Dyn. {\bf 91}, 249--255 (2018).

\bibitem{Lli18b} Llibre, J., Zhang, X.: Limit cycles for discontinuous planar piecewise linear differential systems separated by one straight line and having a center, J. Math. Anal. Appl. {\bf 467}, 537--549 (2018).

\bibitem{LiS19b} Li, S., Llibre, J.: Phase portraits of piecewise linear continuous differential systems with two zones separated by a straight line, J. Differ. Equ. {\bf 266}, 8094--8109 (2019).

\bibitem{Liu10} Liu, X., Han, M.: Bifurcation of limit cycles by perturbing piecewise Hamiltonian systems, Int. J. Bifur. Chaos Appl. Sci. Eng. {\bf 5}, 1--12 (2010).


\bibitem{McK70} McKean Jr., H. P.: Nagumo's equation, Adv. Math. {\bf 4}, 209--223 (1970).

\bibitem{Nag62} Nagumo, J. S., Arimoto, S., Yoshizawa, S.: An active pulse transmission line simulating nerve axon,  Proc. IRE {\bf 50}, 2061--2071 (1962).

\bibitem{Wan16} Wang, Y., Han, M., Constantinesn, D.: On the limit cycles of perturbed discontinuous planar systems with 4 switching lines, Chaos Soliton Fract. {\bf 83}, 158--177 (2016).

\bibitem{Wol20} Wolfram Research, Inc.: Mathematica, Version 12.2, https://www.wolfram.com/mathematica. Champaing, IL (2020).

\bibitem{Xio20} Xiong, Y., Han, M.: Limit cycle bifurcations in discontinuous planar systems with multiple lines, J. Appl. Anal. Comput. {\bf 10}, 361--377 (2020).

\bibitem{Xio21} Xiong, Y., Wang, C.: Limit cycle bifurcations of planar piecewise differential systems with three zones, Nonlin. Anal. Real World Appl. {\bf 61}, 103333, 18 pp (2021).

\bibitem{Yan20} Yang, J.: On the number of limit cycles by perturbing a piecewise smooth Hamilton system with two straight lines of separation, J. Appl. Anal. Comput. {\bf 6}, 2362–2380 (2020).

 
\end{thebibliography}

\end{document}